\documentclass[a4paper, 11pt]{article}
\usepackage{amsmath}
\usepackage{amssymb}
\usepackage{amsthm}
\usepackage{xcolor}
\usepackage[width=460pt, height=650pt]{geometry}
\usepackage{comment}
\usepackage{soul}

\allowdisplaybreaks

\DeclareMathOperator{\N}{\mathbb{N}}
\DeclareMathOperator{\R}{\mathbb{R}}

\DeclareMathOperator{\cB}{\mathcal{B}}
\DeclareMathOperator{\cH}{\mathcal{H}}

\DeclareMathOperator{\iii}{\mathtt{i}}
\DeclareMathOperator{\jjj}{\mathtt{j}}
\DeclareMathOperator{\kkk}{\mathtt{k}}

\DeclareMathOperator{\dimh}{\dim_H}

\renewcommand{\phi}{\varphi}

\numberwithin{equation}{section}

\theoremstyle{plain}
\newtheorem{theorem}{Theorem}[section]

\newtheorem{corollary}[theorem]{Corollary}
\newtheorem{proposition}[theorem]{Proposition}
\newtheorem{lemma}[theorem]{Lemma}

\theoremstyle{remark}

\theoremstyle{definition}

\title{Recurrence rates for shifts of finite type}
\author{Demi Allen$^{1}$ \and Simon Baker$^{2}$ \and Bal\'azs B\'ar\'any$^{3,4,}$\thanks{BB acknowledges support from grants OTKA K123782 and OTKA~FK134251.}}

\begin{document}
\maketitle
\begin{center}
  \vspace{-1em}
$1)$ Department of Mathematics and Statistics, University of Exeter, Harrison Building, North Park Road, Exeter, EX4 4QF, UK.\\
Email address: \texttt{d.d.allen@exeter.ac.uk}\vskip.5em
$2)$Department of Mathematical Sciences, Loughborough University, Loughborough, LE11 3TU, UK\\
Email address: \texttt{simonbaker412@gmail.com}\vskip.5em
$3)$ Department of Stochastics, Institute of Mathematics, Budapest University of Technology and Economics, M\H{u}egyetem rkp. 3., H-1111 Budapest, Hungary.\\
$4)$ MTA-BME Stochastics Research Group, M\H{u}egyetem rkp. 3., H-1111 Budapest, Hungary\\
Email address: \texttt{balubsheep@gmail.com}\vskip.5em
\end{center}
\vspace{1em}

\begin{abstract}
Let $\Sigma_{A}$ be a topologically mixing shift of finite type, let $\sigma:\Sigma_{A}\to\Sigma_{A}$ be the usual left-shift, and let $\mu$ be the Gibbs measure for a H\"{o}lder continuous potential that is not cohomologous to a constant. In this paper we study recurrence rates for the dynamical system $(\Sigma_{A},\sigma)$ that hold $\mu$-almost surely. In particular, given a function $\psi:\N\to \N$ we are interested in the following set $$R_{\psi}=\{\iii\in \Sigma_{A}:i_{n+1}\ldots i_{n+\psi(n)+1}=i_1\ldots i_{\psi(n)}\textrm{ for infinitely many }n\in\N\}.$$

We provide sufficient conditions for $\mu(R_{\psi})=1$ and sufficient conditions for $\mu(R_{\psi})=0$. As a corollary of these results, we discover a new critical threshold where the measure of $R_{\psi}$ transitions from zero to one. This threshold was previously unknown even in the special case of a non-uniform Bernoulli measure defined on the full shift. The proofs of our results combine ideas from Probability Theory and Thermodynamic Formalism. In our final section we apply our results to the study of dynamics on self-similar sets.
\end{abstract}

\section{Introduction}

The notion of recurrence is fundamental in the study of Dynamical Systems and Ergodic Theory. The famous Poincar\'{e} Recurrence Theorem (see, e.g., \cite[Theorem 2.11]{EinsiedlerWard} or \cite[Theorem 1.4]{Wal}) states that if $T:X\to X$ is a measure-preserving transformation on a probability space $(X,\mathcal{B},\mu)$, then for any set $A\in \mathcal{B}$ satisfying $\mu(A)>0,$ we have that $\mu$-almost every $x\in A$ satisfies $T^{n}(x)\in A$ for infinitely many $n\in\mathbb{N}$. Under some modest assumptions this measure theoretic result can be upgraded to a metric one. Indeed if $X$ is equipped with a metric $d$ so that $(X,d)$ is separable and $\mathcal{B}$ is the Borel $\sigma$-algebra, then Poincar\'{e}'s theorem implies that for $\mu$-almost every $x\in X$ we have $$\liminf_{n\to\infty}d(T^{n}(x),x)=0.$$ When $d$ is a metric so that $(X,d)$ is separable and $\mathcal{B}$ is the Borel $\sigma$-algebra, we refer to $(X,\mathcal{B},\mu,T,d)$ as a metric measure-preserving system or a m.m.p.s. It is natural to wonder whether the conclusion $\liminf_{n\to\infty}d(T^{n}(x),x)=0$ for $\mu$-almost every $x$ can be strengthened into something more quantitative. The first result obtained in this direction is the following extremely general quantitative recurrence result obtained by  Boshernitzan in \cite[Theorem 1.2]{Bos}.
\begin{theorem}[Boshernitzan, \cite{Bos}]
	\label{BosTheorem}
Let $(X,\mathcal{B},\mu,T,d)$ be a m.m.p.s. Assume that for some $\alpha >0$ the $\alpha$-dimensional Hausdorff measure $\mathcal{H}^{\alpha}$ is $\sigma$-finite on $(X,d)$. Then for $\mu$-almost every $x\in X$ we have $$\liminf_{n\to\infty}n^{1/\alpha}d(T^n(x),x)<\infty.$$ Moreover, if $\mathcal{H}^{\alpha}(X)=0$, then for $\mu$-almost every $x\in X$ we have $$\liminf_{n\to\infty}n^{1/\alpha}d(T^n(x),x)=0.$$
\end{theorem}
Throughout this paper we denote by $\cH^{\alpha}(X)$ the \emph{$\alpha$-dimensional Hausdorff measure} of a set $X,$ and we write $\dimh{X}$ to denote the \emph{Hausdorff dimension} of $X$. We refer the reader to \cite{Fal} for definitions and further information regarding Hausdorff measures and dimension.

A limitation of Boshernitzan's theorem is that the recurrence rates it provides do not exhibit a dependence on the measure $\mu$, which is contrary to what one would expect. It could well be the case that Theorem~\ref{BosTheorem} does not allow us to conclude an optimal recurrence rate that holds for $\mu$-almost every $x$. This issue was partly addressed in a paper by Barreira and Saussol \cite{BarSau}. In particular they proved the following statement which appears as \cite[Theorem 3]{BarSau}.
\begin{theorem}[Barreira and Saussol, \cite{BarSau}]
	\label{BarSauTheorem}
	If $T:X\to X$ is a Borel measurable map on $X\subset\mathbb{R}^d$, and $\mu$ is a $T$-invariant probability measure on $X$, then for $\mu$-almost every $x\in X,$ we have
	 $$\liminf_{n\to\infty}n^{1/\alpha}d(T^n(x),x)=0\textrm{ for any }\alpha>\liminf_{r\to 0}\frac{\log \mu(B(x,r))}{\log r}.$$
\end{theorem}
Following on from these two important results, two separate research streams have arisen. The first of these streams takes a dynamical system and a recurrence rate, and tries to determine the Hausdorff dimension of the set of points that satisfy this recurrence rate. This line of research was pursued by Tan and Wang in \cite{TanWang} where the underlying dynamical system was the $\beta$-transformation. Seuret and Wang obtained similar results in \cite{SeuWang} for dynamical systems arising from the study of self-conformal sets. The second stream is more measure theoretic in nature. Given a metric measure-preserving system $(X,\mathcal{B},\mu,T,d)$ and a function $\psi:\mathbb{N}\to [0,\infty)$, this stream seeks to find simple criteria which determine the measure of the set
\begin{equation*}
\textit{R}(\psi):=\left\{x\in X:d(T^{n}(x),x)\leq \psi(n)\textrm{ for infinitely many }n\in\mathbb{N} \right\}.
\end{equation*} {This type of problem has been studied in great depth in the context of so called {\it shrinking target problems}. In the shrinking target framework, instead of studying those $x$ satisfying $d(T^{n}(x),x)\leq \psi(n)$ for infinitely many $n$, we fix a $y\in X$ and study those $x$ satisfying $d(T^{n}(x),y)\leq \psi(n)$ for infinitely many $n$. For more on shrinking target problems we refer the reader to \cite{Ath,CheKle,HillVel,HillVel2,KelYu,LLVZW} and the references therein.} Part of the motivation behind the recent activity in this area is to bring the quantitative recurrence theory in line with the theory of shrinking targets. Drawing an analogy with shrinking target problems, it is reasonable to expect that $\mu(R(\psi))$ should be determined by some naturally occurring volume sum. In particular, it is reasonable to expect that there exists a monotone function $f:[0,\infty)\to[0,\infty)$ for which the following holds
\begin{equation}
\label{volume sums}
 \mu(\textit{R}(\psi)) = \left\{ \begin{array}{ll}
1 & \mbox{if $\sum_{n=1}^{\infty}f(\psi(n))=\infty$},\\[2ex]
0 & \mbox{if $\sum_{n=1}^{\infty}f(\psi(n))<\infty$}.\end{array} \right.
\end{equation}
Typically we might expect $f$ to be of the form $f(x)=x^{\gamma}$ where $\gamma>0$ is the Hausdorff dimension of the measure. This principle has been verified in various contexts by several authors. It was shown to be the case for certain natural maps defined on attractors of iterated function systems by Chang et al.~in \cite{ChangWuWu}, and by the second author and Farmer in \cite{BakFar}. Hussain et al.~in \cite{HLSW} gave general conditions for a dynamical system to guarantee that \eqref{volume sums} holds. This result applies to many well known dynamical systems including the $\beta$-transformation and the Gauss map. Further refinements were obtained by Kirsebom et al.~in \cite{KKP}. They obtained positive results for a more general class of dynamical system than was previously considered in \cite{BakFar,ChangWuWu,HLSW}. They also considered a more general notion of recurrence where the function $\psi$ was also allowed to depend upon the point $x$. Recently Kleinbock and Zheng proved a general result that implies a zero-one law for both the shrinking target problem and the quantitative recurrence problem \cite{KleZhe}. This result applies to a general family of expanding dynamical systems satisfying suitable bounded distortion assumptions and for which the underlying measure satisfies the following property: there exist constants $c_1,c_2,\alpha>0$ such that for any $x\in X$ and $r\in(0,Diam(X))$ we have 
\[c_1r^{\alpha}\leq \mu(B(x,r))\leq c_2r^{\alpha}.\] Measures satisfying this property are called \emph{Ahlfors regular measures}. 
We conclude this overview of related works by mentioning a paper by Chazottes and Ugalde \cite{ChaUg}. Just as we do in this paper, the authors of \cite{ChaUg} study Gibbs measures for shifts of finite type. In their paper they obtained almost sure bounds for the first time a sequence returns under the left-shift to the cylinder determined by its first $n$ entries. For certain functions $\psi$ these bounds can be used to obtain measure statements for $\textit{R}(\psi)$. However, the class of $\psi$ for which this can be done is quite restrictive and does not contain many natural choices of $\psi$ one would be interested in. This is not surprising given that the bounds stated in \cite{ChaUg} apply to all first returns, whereas $\textit{R}(\psi)$ is defined in terms of a subsequence of returns. In this paper we obtain detailed information on the measure of the set $R(\psi)$ that goes beyond that which can be derived from \cite{ChaUg}.

In this paper we study recurrence rates for shifts of finite type that hold $\mu$-almost surely with respect to a Gibbs measure $\mu$. Of special interest will be the case where {the defining potential of $\mu$ has positive variance with respect to $\mu$. For short, we say that {\it $\mu$ has positive variance} in this case.} This seemingly innocuous assumption leads to a much richer theory and a wider range of behaviour than was previously observed in \cite{BakFar,ChangWuWu,HLSW,KKP,KleZhe}. The main results of these papers each assume that the underlying measure $\mu$ satisfies some additional uniformity assumption. Often this uniformity assumption is that the measure is Ahlfors regular. This assumption is very useful for technical reasons, and often means that one can study the measure of $R(\psi)$ using a similar toolkit to that which one would use to study shrinking target problems. However, this uniformity assumption rules out measures $\mu$ with positive variance. These measures are often highly non-uniform and exhibit strong multifractal behaviour. Importantly, it is precisely this non-uniformity that leads to the richer theory and the wider range of behaviour mentioned above. We will properly formalise our results in the next section. For now we state the following result which illustrates the phenomenon described in this paragraph. This result is a special case of Theorem \ref{loglogtheorem}.
\begin{theorem}
	\label{Example theorem}
Let $(p_i)_{i=1}^{K}$ be a non-uniform\footnote{$p_i\neq 1/K$ for some $i$. This condition ensures positive variance.} probability vector and let $\mu$ be the corresponding Bernoulli measure defined on the full shift $\Sigma=\{1,\ldots,K\}^{\mathbb{N}}$. Equip $\Sigma$ with the metric $d$ given by $d(\iii,\jjj)=K^{-|\iii\wedge\jjj|}$ where $|\iii\wedge\jjj|:=\inf\{n\geq0:i_{n+1}\neq j_{n+1}\}$. Then there exist constants $h,\rho>0$ such that the following holds: for  $\varepsilon>0$  let $\psi_{\varepsilon}^{+}:\mathbb{N}\to\mathbb{N}$ and $\psi_{\varepsilon}^{-}:\mathbb{N}\to\mathbb{N}$ be given by: $$\psi_{\varepsilon}^{+}(n)=\left\lfloor \frac{\log n}{h}+\frac{(1+\varepsilon)}{h^{3/2}}\sqrt{2\rho\log n\log \log \log n}\right \rfloor$$ and $$\psi_{\varepsilon}^{-}(n)=\left\lfloor \frac{\log n}{h}+\frac{(1-\varepsilon)}{h^{3/2}}\sqrt{2\rho\log n\log \log \log n}\right \rfloor.$$ We then let $\Psi_{\varepsilon}^{+}:\mathbb{N}\to[0,\infty)$ and $\Psi_{\varepsilon}^{-}:\mathbb{N}\to[0,\infty)$ be given by $\Psi_{\varepsilon}^{+}(n)=K^{-\psi_{\varepsilon}^{+}(n)}$ and \mbox{$\Psi_{\varepsilon}^{-}(n)=K^{-\psi_{\varepsilon}^{-}(n)}$}. Then for any $\varepsilon>0$ we have $\mu(R(\Psi_{\varepsilon}^{+}))=0$ and $\mu(R(\Psi_{\varepsilon}^{-}))=1.$
\end{theorem}
Theorem \ref{Example theorem} demonstrates a critical threshold which to the best of our knowledge has not previously been observed in the study of quantitative recurrence or shrinking target problems. We also emphasise that unlike in \cite{BakFar,ChangWuWu,HLSW,KKP,KleZhe}, the presence of the $\sqrt{2\rho\log n\log \log \log n}$ term appearing in Theorem \ref{Example theorem} means that there is no $\gamma>0$ such that for every $\psi:\mathbb{N}\to\mathbb{N}$ we have \eqref{volume sums} for some $f$ of the form $f(x)=x^{\gamma}$. \\

\noindent \textbf{Notation.} We end this introductory section by formalising some notation that we will use throughout. Given two positive real valued functions $f,g:S\to(0,\infty)$ defined on some set $S$, we write $f\ll g$ if there exists a constant $C>0$ such that $f(x)\leq Cg(x)$ for all $x\in S$. For each $k\in\mathbb{N}$ we write $\log^{(k)}$ to denote the function that is the $k$-fold composition of $\log$ with itself, e.g. $\log^{(3)}(x)=\log \log \log x.$ To avoid additional notational complexities, we will adopt the convention throughout that $\log x=0$ whenever $x\leq 1$.

\section{Background and statement of results}
\subsection{Background on Thermodynamic Formalism}
Let $K\geq2$ be an integer and let
\[\Sigma=\{1,\ldots,K\}^{\N}\]
be the set of infinite sequences formed by elements of the set $\{1,\ldots,K\}.$ For integers $n \geq 0$, let us denote by
\[\Sigma_n=\{1,\ldots,K\}^n\]
the set of words of length $n$, by convention $\Sigma_0=\emptyset$. We also let
\[\Sigma_*=\bigcup_{n=0}^\infty\Sigma_n\]
be the set of all finite words. For $\iii\in\Sigma_*$, denote by $|\iii|$ the length of $\iii$. For $\iii,\jjj\in\Sigma\cup\Sigma_*$, denote by $\iii\wedge\jjj$ the common prefix of $\iii$ and $\jjj$; that is, let
\[|\iii\wedge\jjj|:=\inf\{n\geq0:i_{n+1}\neq j_{n+1}\} \qquad \text{and} \qquad \iii\wedge\jjj=(i_1,\ldots,i_{|\iii\wedge\jjj|}).\]
Note that if $|\iii\wedge\jjj|=0$ then we define $\iii\wedge\jjj$ as the empty word, and if $|\iii\wedge\jjj|=\infty$ or $|\iii|=|\jjj|=|\iii \wedge \jjj|$ then $\iii\wedge\jjj=\iii=\jjj$. We equip $\Sigma$ with the metric given by
\[d(\iii,\jjj)=K^{-|\iii\wedge\jjj|}.\]
For a finite word $\iii\in\Sigma_*$ and for a sequence $\jjj\in\Sigma\cup\Sigma_*$, we denote by $\iii\jjj$ the concatenation of $\iii$ and $\jjj$. Moreover, for $\iii\in \Sigma_*$ and $B\subseteq\Sigma_*$, let
\[\iii B=\{\iii\jjj:\jjj\in B\}.\] For a word $\iii\in\Sigma\cup\Sigma_*$ and $1\leq n\leq m\leq|\iii|$, we let
\[\iii|_n^m=(i_n,\ldots,i_m).\]
Let $\sigma\colon\Sigma\mapsto\Sigma$ be the left-shift map; that is,
\[\sigma(i_1,i_2,\ldots)=(i_2,i_3,\ldots).\]

Given a $K\times K$ matrix $A$ consisting of zeros and ones we can define the corresponding \emph{shift of finite type} as follows:
$$\Sigma_{A}:=\left\{\iii\in \Sigma:A_{i_l,i_{l+1}}=1\textrm{ for all }l\in \N\right\}.$$ Notice that $\sigma$ is well defined as a map from $\Sigma_{A}$ to $\Sigma_{A}.$ Also, note that the full shift corresponds to the case when the matrix $A$ consists entirely of ones. We will always assume that there exists $M\in\mathbb{N}$ such that each entry of $A^{M}$ is strictly positive. This assumption means that the map $\sigma:\Sigma_{A}\to\Sigma_{A}$ is topologically mixing (for the definition of topologically mixing see \cite[Definition 1.8.2]{KatHas}). For each $n\in \N$ we let
$$\Sigma_{A,n}:=\left\{\iii\in \Sigma_n:A_{i_{l},i_{l+1}}=1\textrm{ for all }1\leq l\leq n-1 \right\}$$ and write
\[\Sigma_{A,*}=\bigcup_{n=1}^{\infty}\Sigma_{A,n}.\]
For an $\iii\in\Sigma_{A,*}$, let $[\iii]$ be the corresponding cylinder set of sequences in $\Sigma_{A}$ that begin with $\iii$, i.e. $$[\iii]:=\{\jjj\in\Sigma_{A}:\iii=j_1\ldots j_{|\iii|}\}.$$ For a subset $B\subseteq\Sigma_{A,*}$ we also let $[B]:=\bigcup_{\iii\in B}[\iii]$.

We call a map $f\colon\Sigma_{A}\mapsto\R$ \emph{H\"older-continuous} if there exist constants $b>0$ and $0<\alpha<1$ such that
\[|f(\iii)-f(\jjj)|\leq b\alpha^{|\iii\wedge\jjj|}\]
for all $\iii,\jjj\in \Sigma_{A}$. Let us define the \emph{pressure} of such an $f$ in the usual way by
$$
P(f)=\lim_{n\to\infty}\frac{1}{n}\log\left(\sum_{\iii\in\Sigma_{A,n}}\sup_{\jjj\in[\iii]}\exp\left(\sum_{k=0}^{n-1}f(\sigma^k\jjj)\right)\right).
$$
By \cite[Theorem~1.4, Proposition 1.13, and Theorem 1.16]{Bowen} there exists a unique $\sigma$-invariant ergodic measure $\mu$ on $\Sigma_{A}$ for which there exists a constant $C>1$ such that
\begin{equation}\label{eq:Gibbs}
C^{-1}\leq\dfrac{\mu([i_1,\ldots,i_n])}{\exp\left(-nP(f)+\sum_{k=0}^{n-1}f(\sigma^k\iii)\right)}\leq C
\end{equation}
for every $\iii\in \Sigma_A$ and $n\in\mathbb{N}$. We call  $\mu$ the {\it Gibbs measure} of the potential $f$. For the purpose of exposition, we suppress the dependence of $\mu$ on $f$ throughout. It can be seen from the definition of Gibbs measure that there exists a constant $C>1$ such that for every $\iii,\jjj\in\Sigma_{A,*}$ satisfying $\iii\jjj\in\Sigma_{A,*},$ we have
\begin{equation}\label{eq:quasiBernoulli}
	C^{-1}\leq\frac{\mu([\iii\jjj])}{\mu([\iii])\mu([\jjj])}\leq C.
\end{equation}
Without loss of generality, we may assume that the constants in \eqref{eq:Gibbs} and \eqref{eq:quasiBernoulli} are equal.

We define the \emph{entropy} of a Gibbs measure $\mu$ by
$$
h_\mu:=\lim_{n\to\infty}\frac{-1}{n}\sum_{\iii\in\Sigma_{A,n}}\mu([\iii])\log\mu([\iii]).
$$
It can be shown, see \cite[Theorem 1.22]{Bowen} for example, that
\begin{align} \label{variational principle}
h_\mu&=P(f)-\int f(\iii)d\mu(\iii).
\end{align}
We define the \emph{variance} of a Gibbs measure $\mu$ of some potential $f$ to be $$\rho_{\mu}:=\lim_{n\to\infty}\frac{1}{n}\int \left(\sum_{k=0}^{n-1}f(\sigma^{k}\iii)-n\int f(\jjj)\, d\mu(\jjj)\right)^2\, d\mu(\iii).$$ We emphasise that this limit always exists (see \cite{PP}). Clearly $\rho_{\mu}\geq 0$. It is well known that we have equality here if and only if $f$ is cohomologous to a constant, i.e. there exists a continuous function $g:\Sigma_{A}\to\mathbb{R}$ and $c\in\mathbb{R}$ such that $f=g\circ \sigma -g +c$. For a proof see \cite{PP}. We will often assume that $f$ is not cohomologous to a constant and therefore that $\rho_{\mu}>0$.

If $(p_i)_{i=1}^{K}$ is a probability vector and $\mu$ is the corresponding Bernoulli measure on the full shift $\Sigma,$ then $\mu$ can be realised as the Gibbs measure for the potential $f:\Sigma \to \mathbb{R}$ given by $f(\iii)=\log p_{i_1}$. We recall here that the Bernoulli measure corresponding to the probability vector $(p_i)_{i=1}^{K}$ is defined by assigning cylinders measure
\[\mu([\iii]) = p_{i_1} p_{i_1} \dots p_{i_{|\iii|}}\]
for each $\iii \in \Sigma_*$. An important consequence of this definition is that for any two words $\iii, \jjj \in \Sigma_*$, we have
\begin{align} \label{Bernoulli}
\mu([\iii\jjj]) = \mu([\iii])\mu([\jjj]).
\end{align}
For further details regarding the definition of Bernoulli measures and some of their properties we refer the reader to \cite[Chapter 4.2]{KatHas} and \cite{Wal}. We emphasise that when $\mu$ is a Bernoulli measure with corresponding probability vector $(p_i)_{i=1}^{K},$ then the entropy and variance of $\mu$ take the following simpler form:
$$h_\mu=-\sum_{i=1}^{K}p_i\log p_i$$ and $$\rho_{\mu}=\sum_{i=1}^{K}p_i(\log p_i)^2-\left(\sum_{i=1}^{K}p_i\log p_i\right)^2.$$ 
It can be seen that the potential $f$ corresponding to a Bernoulli measure $\mu$ is cohomologous to a constant if and only if $f$ is a constant function. Therefore $\rho_{\mu}>0$ if and only if $(p_i)_{i=1}^{K}$ is not the uniform probability vector.

\subsection{Statement of results}
Let $\psi\colon\N\mapsto\N$ and let $R_\psi$ be the set of infinite sequences that return infinitely often to the neighbourhood determined by $\psi$ when we apply the left shift. That is, let
\begin{equation}\label{eq:recset}
R_\psi:=\{\iii\in\Sigma_{A}:d(\sigma^n\iii,\iii)\leq K^{-\psi(n)}\text{ for infinitely many }n\in\mathbb{N}\}.
\end{equation} Note that this definition is slightly different to the definition of $\textit{R}(\psi)$ given earlier. This definition helps to simplify much of our analysis. In our subsequent proofs we will often make the assumption that $\psi$ satisfies $\psi(n)\leq n$ for all $n$ sufficiently large. Under this assumption $R_{\psi}$ has the following simple formulation
\begin{equation}
\label{eq:simpleform}
R_{\psi}=\bigcap_{n=1}^\infty\bigcup_{m=n}^\infty\bigcup_{\substack{\iii\in\Sigma_{A,\psi(m)}\\\jjj\in\Sigma_{A,m-\psi(m)}\\ \iii\jjj\iii\in \Sigma_{A,m+\psi(m)}}}[\iii\jjj\iii].
\end{equation}
The assumption that $\psi$ satisfies $\psi(n)\leq n$ for all $n$ sufficiently large is not especially restrictive. Indeed for any Gibbs measure $\mu$ it can be shown using the Borel--Cantelli Lemma (see Lemma~\ref{lem:BC}) and the Shannon--McMillan--Breiman Theorem \cite[Chapter 6, Theorem 2.3]{Petersen} that
$$\mu\left(\{\iii\in\Sigma_{A}:d(\sigma^k\iii,\iii)\leq K^{-k}\text{ for infinitely many }k\in\mathbb{N}\}\right)=0.$$
As such we can often assume without loss of generality that $\psi$ satisfies $\psi(n)\leq n$ for all $n$ sufficiently large.

The goal of this paper is to find “close to optimal" conditions for establishing a ``zero-one law'' for the measure of the set $R_\psi$. The following theorem provides sufficient conditions for $R_{\psi}$ to have measure zero.

\begin{theorem}\label{thm:convergence}
	Let $\mu$ be the Gibbs measure for a H\"older continuous potential that is not cohomologous to a constant. Let $\psi\colon\N\mapsto\N$ be such that there exists $\varepsilon>0$ for which
	$$
		\sum_{n=1}^\infty e^{-h_\mu\psi(n)+(1+\varepsilon)\sqrt{2\rho_\mu\psi(n)\log\log\psi(n)}}<\infty.
		$$
		Then $\mu(R_\psi)=0$.
\end{theorem}

Given $\psi:\mathbb{N}\to\mathbb{N}$ and $n\in\N$, let $\psi^{-1}(n)=\{m\in\N:\psi(m)=n\}$. For any set $A\subseteq\N$ we denote the cardinality of $A$ by $\#A$. The following theorem is the main result of this paper. It gives sufficient conditions for $\mu(R_{\psi})=1$.

\begin{theorem}\label{thm:meas1}
Let $\mu$ be the Gibbs measure for a H\"older continuous potential that is not cohomologous to a constant. Let $\psi\colon\N\mapsto\N$ be such that $\psi(n)\leq n$ for all $n$ sufficiently large and $\lim_{n\to\infty}\psi(n)=\infty$. Suppose that there exists an increasing sequence $(n_k)_{k=1}^{\infty}$ and a function $g\colon\N\mapsto[1,\infty)$ satisfying $\lim\limits_{n\to\infty}\frac{g(n)}{g(n+1)}=1,$ such that
	\begin{equation}\label{eq:cond1}
	\mu\left(\left\{\iii\in\Sigma_{A}:\limsup_{k\to\infty}\frac{\log\mu([\iii|_1^{n_k}])+h_\mu n_k}{\sqrt{2\rho_\mu n_k}g(n_k)}>1\right\}\right)=1
\end{equation}
	and
	\begin{equation}\label{eq:cond2}
\lim_{k\to\infty}\#\psi^{-1}(n_k)e^{-h_\mu n_k+\sqrt{2\rho_\mu n_k} g(n_k)}=\infty.
\end{equation}
Then $$\mu\left(\left\{\iii\in \Sigma_{A}:d(\sigma^{p}\iii,\iii)\leq K^{-n_{k}}\textrm{ for some }p\in \psi^{-1}(n_k) \textrm{ for infinitely many }k\in\mathbb{N}\right\}\right)=1.$$
Moreover, $\mu(R_\psi)=1$. \end{theorem}
We emphasise that Theorem \ref{thm:meas1} also holds without the assumption $\lim_{n\to\infty}\psi(n)=\infty.$ It can be shown using the arguments given in this paper that whenever $\liminf_{n\to\infty}\psi(n)<\infty$ we have $\mu(R_{\psi})=1$. Theorem \ref{thm:meas1} implies the following corollary.
\begin{corollary}
\label{cor:Applications corollary}
Let $\mu$ be the Gibbs measure for a H\"older continuous potential that is not cohomologous to a constant. Let $\psi\colon\N\mapsto\N$ be such that $\psi(n)\leq n$ for all $n$ sufficiently large and $\lim_{n\to\infty}\psi(n)=\infty$. Suppose that there exists $\varepsilon>0$ for which
\begin{equation}
\label{cardinality bound}
\left\lceil e^{h_\mu n-(1-\varepsilon)\sqrt{2\rho_\mu n\log\log{n}}}\right\rceil\ll \#\psi^{-1}(n)
\end{equation}
for all $n$ sufficiently large. Then $\mu(R_{\psi})=1$.
\end{corollary}
\begin{proof}
We take $(n_k)_{k=1}^{\infty}=(k)_{k=1}^{\infty}$ and $g(n)=(1-\varepsilon/2)\sqrt{\log \log n}.$ Then, by \eqref{cardinality bound}, we have that
\[\# \psi^{-1}(n_k) e^{-h_{\mu}n_k + \sqrt{2\rho_{\mu}n_k}g(n_k)} \gg \exp\left(\frac{\varepsilon}{2}\sqrt{2\rho_{\mu}k\log\log{k}}\right)\]
for all $k$ sufficiently large and so \eqref{eq:cond2} follows readily. The fact assumption \eqref{eq:cond1} holds is a consequence of the Law of the Iterated Logarithm for Gibbs measures, see Theorem \ref{thm:LIL} below. By Theorem \ref{thm:meas1} we now have $\mu(R_{\psi})=1$.
\end{proof}
One case our analysis does not cover is when $\psi$ only satisfies the weaker assumption that $$\left\lceil e^{h_\mu n-\sqrt{2\rho_\mu n\log\log{n}}}\right\rceil\ll \#\psi^{-1}(n)$$ for all $n$ sufficiently large. It seems to be a challenging problem to determine $\mu(R_\psi)$ under this assumption.

Combining Theorem \ref{thm:convergence} and Corollary \ref{cor:Applications corollary} we obtain the following result. It identifies a critical threshold for the measure of $R_{\psi}$. This result implies Theorem \ref{Example theorem} stated in the introduction.

\begin{theorem}
	\label{loglogtheorem}
Let $\mu$ be the Gibbs measure for a H\"older continuous potential that is not cohomologous to a constant. For $\varepsilon>0$ let $\psi_{\varepsilon}^{+}:\mathbb{N}\to\mathbb{N}$ and $\psi_{\varepsilon}^{-}:\mathbb{N}\to\mathbb{N}$ be given by: $$\psi_{\varepsilon}^{+}(n)=\left\lfloor \frac{\log n}{h_\mu}+\frac{(1+\varepsilon)}{h_{\mu}^{3/2}}\sqrt{2\rho_{\mu}\log n\log \log \log n}\right \rfloor$$ and $$\psi_{\varepsilon}^{-}(n)=\left\lfloor \frac{\log n}{h_\mu}+\frac{(1-\varepsilon)}{h_{\mu}^{3/2}}\sqrt{2\rho_{\mu}\log n\log \log \log n}\right \rfloor.$$ Then for any $\varepsilon>0$ we have $\mu(R_{\psi_{\varepsilon}^{+}})=0$ and $\mu(R_{\psi_{\varepsilon}^{-}})=1.$
\end{theorem}

\begin{proof}
This result is a consequence of Theorem \ref{thm:convergence} and Corollary \ref{cor:Applications corollary}. We split our proof into two parts. \\

\noindent \textbf{Proof that $\mu(R_{\psi_{\varepsilon}^{+}})=0.$}
Fix $\varepsilon>0$. We begin by noting that for all $n$ sufficiently large we have
\begin{equation}
\label{bound1}
\psi_{\varepsilon}^{+}(n)=\left\lfloor \frac{\log n}{h_{\mu}}+\frac{(1+\varepsilon)}{h_{\mu}^{3/2}}\sqrt{2\rho_{\mu}\log n\log \log \log n}\right \rfloor\leq \left(1+\frac{\varepsilon}{2}\right)^{2/3}\frac{\log n}{h_{\mu}}
\end{equation}
and
\begin{equation}
\label{bound2}
\log \log\left(\left(1+\frac{\varepsilon}{2}\right)^{2/3}\frac{\log n}{h_{\mu}}\right)\leq \left(1+\frac{\varepsilon}{2}\right)^{2/3}\log \log \log n.
\end{equation}
Writing $e(x)=e^x$, it follows from the definition of $\psi_{\varepsilon}^{+}$ that
\begin{align*}
&\sum_{n=1}^{\infty} e^{-h_{\mu}\psi_{\varepsilon}^{+}(n)+(1+\varepsilon/2)^{1/3}\sqrt{2\rho_{\mu} \psi_{\varepsilon}^{+}(n)\log \log \psi_{\varepsilon}^{+}(n)}}\\
\ll &\sum_{n=1}^{\infty}\frac{1}{n}e\left(-\frac{(1+\varepsilon)}{h_{\mu}^{1/2}}\sqrt{2\rho_{\mu} \log n\log^{(3)} n}+\left(1+\frac{\varepsilon}{2}\right)^{1/3}\sqrt{2\rho_{\mu} \psi_{\varepsilon}^{+}(n)\log^{(2)} \psi_{\varepsilon}^{+}(n)}\right).
\end{align*}
Next, \eqref{bound1} gives us that
\begin{align*}
&\sum_{n=1}^{\infty} e^{-h_{\mu}\psi_{\varepsilon}^{+}(n)+(1+\varepsilon/2)^{1/3}\sqrt{2\rho_{\mu} \psi_{\varepsilon}^{+}(n)\log \log \psi_{\varepsilon}^{+}(n)}}\\
\ll &\sum_{n=1}^{\infty}\frac{1}{n}e\left(-\frac{(1+\varepsilon)}{h_{\mu}^{1/2}}\sqrt{2\rho_{\mu} \log n\log^{(3)} n}+\left(1+\frac{\varepsilon}{2}\right)^{1/3}\sqrt{2\rho_{\mu} \left(1+\frac{\varepsilon}{2}\right)^{2/3}\frac{\log n}{h_{\mu}}\log^{(2)} \left(\left(1+\frac{\varepsilon}{2}\right)^{2/3}\frac{\log n}{h_{\mu}}\right)}\right).
\end{align*}
Finally, it follows from \eqref{bound2} that
\begin{align*}
&\sum_{n=1}^{\infty} e^{-h_{\mu}\psi_{\varepsilon}^{+}(n)+(1+\varepsilon/2)^{1/3}\sqrt{2\rho_{\mu} \psi_{\varepsilon}^{+}(n)\log \log \psi_{\varepsilon}^{+}(n)}}\\
\stackrel{\eqref{bound2}}{\ll} &\sum_{n=1}^{\infty}\frac{1}{n}e\left(-\frac{(1+\varepsilon)}{h_{\mu}^{1/2}}\sqrt{2\rho_{\mu} \log n\log^{(3)} n}+\frac{(1+\varepsilon/2)}{h_{\mu}^{1/2}}\sqrt{2\rho_{\mu} \log n\log^{(3)} n}\right)\\
=&\sum_{n=1}^{\infty}\frac{1}{n}e\left(-\frac{\varepsilon/2}{h_{\mu}^{1/2}}\sqrt{2\rho_{\mu} \log n\log^{(3)} n}\right)\\
<&\infty.
\end{align*}
Therefore by Theorem \ref{thm:convergence} we have that $\mu(R_{\psi_{\varepsilon}^{+}})=0$.\\

\noindent \textbf{Proof that $\mu(R_{\psi_{\varepsilon}^{-}})=1.$} Fix $0 < \varepsilon < 1$. Define $\tilde{\psi}^{-}_{\varepsilon}:[0,\infty)\to [0,\infty)$ by $$\tilde{\psi}^{-}_{\varepsilon}(x)=\frac{\log x}{h_{\mu}}+\frac{(1-\varepsilon)}{h_{\mu}^{3/2}}\sqrt{2\rho_{\mu}\log x\log^{(3)} x}
.$$ Note that $\lfloor \tilde{\psi}^{-}_{\varepsilon}(n)\rfloor = \psi^{-}_{\varepsilon}(n)$ for all $n\in\mathbb{N}$. It is a straightforward albeit tedious calculation to verify that there exist constants $C_1,C_2>0$ such that
\begin{equation}
\label{derivative bound}
\frac{C_1}{x}\leq (\tilde{\psi}^{-}_{\varepsilon})'(x)\leq \frac{C_2}{x}
\end{equation} for all $x$ sufficiently large. It can also be shown that $(\tilde{\psi}^{-}_{\varepsilon})':[0,\infty)\to[0,\infty)$ is decreasing. We define $m(n)\in[0,\infty)$ implicitly as the unique solution to the equation $\tilde{\psi}^{-}_{\varepsilon}(m(n))=n.$ Using the Mean Value Theorem together with \eqref{derivative bound} and the fact that $(\tilde{\psi}^{-}_{\varepsilon})'$ is decreasing, we may deduce the following
\begin{align*}
\#(\psi^{-}_{\varepsilon})^{-1}(n)=\# \{m\in\mathbb{N}:\tilde{\psi}^{-}_{\varepsilon}(m)\in[n,n+1)\}
&\gg |m(n+1)-m(n)|\\
&\geq ((\tilde{\psi}^{-}_{\varepsilon})'(m(n)))^{-1}\\
&\gg \frac{m(n)}{C_2}.
\end{align*}  It follows from the above that to obtain the cardinality bounds required by Corollary \ref{cor:Applications corollary}, it is sufficient to obtain a good lower bound for $m(n)$. By the definition of $m(n)$ and using the fact that $\tilde{\psi}^{-}_{\varepsilon}$ is increasing, it follows that to obtain the desired bound it is sufficient to show that $$\tilde{\psi}^{-}_{\varepsilon}(e^{h_\mu n-(1-\frac{\varepsilon}{2})\sqrt{2\rho_\mu n\log^{(2)}{n}}})< n$$ for all $n$ sufficiently large. This we do below.

The following holds for all $n$ sufficiently large:
\begin{align*}
&\tilde{\psi}^{-}_{\varepsilon}(e^{h_\mu n-(1-\frac{\varepsilon}{2})\sqrt{2\rho_\mu n\log^{(2)}{n}}})\\
=&n-\frac{(1-\frac{\varepsilon}{2})}{h_{\mu}}\sqrt{2\rho_{\mu}n\log^{(2)} n}+\frac{(1-\varepsilon)}{h_{\mu}^{3/2}}\sqrt{2\rho_{\mu}\log \left(e^{h_\mu n-(1-\frac{\varepsilon}{2})\sqrt{2\rho_\mu n\log^{(2)}{n}}}\right)\log^{(3)} \left(e^{h_\mu n-(1-\frac{\varepsilon}{2})\sqrt{2\rho_\mu n\log^{(2)}{n}}}\right)}\\
=&n-\frac{(1-\frac{\varepsilon}{2})}{h_{\mu}}\sqrt{2\rho_{\mu}n\log^{(2)} n}+\frac{(1-\varepsilon)}{h_{\mu}^{3/2}}\sqrt{2\rho_{\mu}\left(h_\mu n-\left(1-\frac{\varepsilon}{2}\right)\sqrt{2\rho_\mu n\log^{(2)}{n}}\right)\log^{(3)} \left(e^{h_\mu n-(1-\frac{\varepsilon}{2})\sqrt{2\rho_\mu n\log^{(2)}{n}}}\right)}\\
\leq & n-\frac{(1-\frac{\varepsilon}{2})}{h_{\mu}}\sqrt{2\rho_{\mu}n\log^{(2)} n}+\frac{(1-\varepsilon)}{h_{\mu}^{3/2}}\sqrt{2\rho_{\mu}h_{\mu}n\log^{(3)} \left(e^{h_\mu n-(1-\frac{\varepsilon}{2})\sqrt{2\rho_\mu n\log^{(2)}{n}}}\right)}\\
= & n-\frac{(1-\frac{\varepsilon}{2})}{h_{\mu}}\sqrt{2\rho_{\mu}n\log^{(2)} n}+\frac{(1-\varepsilon)}{h_{\mu}}\sqrt{2\rho_{\mu}n\log^{(3)} \left(e^{h_\mu n-(1-\frac{\varepsilon}{2})\sqrt{2\rho_\mu n\log^{(2)}{n}}}\right)}\\
\leq & n-\frac{(1-\frac{\varepsilon}{2})}{h_{\mu}}\sqrt{2\rho_{\mu}n\log^{(2)} n}+\frac{(1-\varepsilon)}{h_{\mu}}\sqrt{2\rho_{\mu}n\log^{(3)} \left(e^{h_\mu n}\right)}\\
=& n-\frac{(1-\frac{\varepsilon}{2})}{h_{\mu}}\sqrt{2\rho_{\mu}n\log^{(2)} n}+\frac{(1-\varepsilon)}{h_{\mu}}\sqrt{2\rho_{\mu}n\log^{(2)} h_\mu n}\\
\leq & n-\frac{(1-\frac{\varepsilon}{2})}{h_{\mu}}\sqrt{2\rho_{\mu}n\log^{(2)} n}+\frac{(1-\frac{3\varepsilon}{4})}{h_{\mu}}\sqrt{2\rho_{\mu}n\log^{(2)} n}\\
=&n-\frac{\frac{\varepsilon}{4}}{h_{\mu}}\sqrt{2\rho_{\mu}n\log^{(2)} n}\\
<&n.
\end{align*}
It follows from the above that $$m(n)\geq e^{h_\mu n-(1-\varepsilon/2)\sqrt{2\rho_\mu n\log\log{n}}}$$ for all $n$ sufficiently large. This in turn implies that
$$\#(\psi^{-}_{\varepsilon})^{-1}(n)\gg e^{h_\mu n-(1-\varepsilon/2)\sqrt{2\rho_\mu n\log\log{n}}}$$ for all $n$ sufficiently large. Therefore the function $\psi^{-}_{\varepsilon}$ satisfies the bound \eqref{cardinality bound} and by Corollary~\ref{cor:Applications corollary} we may conclude that $\mu(R_{\psi^{-}_{\varepsilon}})=1.$

\end{proof}

The following result gives a divergence condition in the form of \eqref{volume sums} which guarantees that $R_{\psi}$ has full measure.
\begin{corollary}
	Let $\mu$ be the Gibbs measure for a H\"older continuous potential that is not cohomologous to a constant. Let $\psi\colon\N\mapsto\N$ be such that $\psi(n)\leq n$ for all $n\in\N$ and $\lim_{n\to\infty}\psi(n)=\infty$. Suppose that there exists a constant $K>0$ for which
	$$
	\sum_{n=1}^\infty e^{-h_\mu\psi(n)+K\sqrt{2\rho_\mu\psi(n)}}=\infty.
	$$
	Then $\mu(R_\psi)=1$.
\end{corollary}

\begin{proof}
	By our divergence assumption we have
	$$
	\sum_{n=1}^{\infty}\#\psi^{-1}(n)e^{-h_\mu n+K\sqrt{2\rho_\mu n}}=\sum_{n=1}^{\infty}e^{-h_\mu \psi(n)+K\sqrt{2\rho_\mu\psi(n)}}=\infty.
	$$	
	Hence, there exists a sequence $(n_k)_{k=1}^{\infty}$ such that
	\begin{equation}
	\label{cardinality lower bound}
	\#\psi^{-1}(n_k)e^{-h_\mu n_k+K\sqrt{2\rho_\mu n_k}}>n_k^{-2}
	\end{equation} holds for all $k \in \N$. Consider the set $$B:=\left\{\iii\in \Sigma_{A}:	\limsup\limits_{k\to\infty}\frac{\log\mu([\iii|_1^{n_k}])+h_\mu n_k}{\sqrt{2\rho_\mu n_k}}=\infty\right\}.$$ Appealing to the fact that $\mu$ satisfies \eqref{eq:quasiBernoulli}, it can be shown that being an element of $B$ is independent of any initial string of digits. Therefore $B$ is a \emph{tail event}, i.e. $B$ belongs to the tail $\sigma$-algebra given by $\bigcap_{n=0}^{\infty}{\sigma^{-n}\cB^+}$ where $\cB^+$ is the Borel $\sigma$-algebra on $\Sigma_{A,*}$. It is well known that Gibbs measures for shifts of finite type are exact, i.e. the tail $\sigma$-algebra is the trivial $\sigma$-algebra (see \cite[Remark 3, Page 29]{PP}). It follows from exactness that either $\mu(B)=0$ or $\mu(B)=1$. Let us suppose that $\mu(B)=0$. Then there exists $M\in\mathbb{R}$ sufficiently large for which $$B_{M}=\left\{\iii\in \Sigma_{A}:	\limsup\limits_{k\to\infty}\frac{\log\mu([\iii|_1^{n_k}])+h_\mu n_k}{\sqrt{2\rho_\mu n_k}}\leq M\right\}$$ has positive measure. Importantly $B_{M}$ is also a tail event. This can also be shown using \eqref{eq:quasiBernoulli}. Therefore by exactness we must have $\mu(B_{M})=1$. However this contradicts the Central Limit Theorem for Gibbs measures, see Theorem~\ref{thm:CLT}. Therefore we must have $\mu(B)=1$.

	Now fix an arbitrary $K' \in \N$ with $K'>K$ and define $g:\mathbb{N}\to\mathbb{N}$ according to the rule $g(n) = K'$. Using the fact that $\mu(B)=1$ we may conclude that \eqref{eq:cond1} from Theorem \ref{thm:meas1} is satisfied for this choice of $g$. Using \eqref{cardinality lower bound} we have that
	$$
	\#\psi^{-1}(n_k)e^{-h_\mu n_k+K'\sqrt{2\rho_\mu n_k}}>\frac{e^{(K'-K)\sqrt{2\rho_\mu n_k}}}{n_k^2}\to\infty\text{ as }k\to\infty.
	$$
	Therefore \eqref{eq:cond2} from Theorem~\ref{thm:meas1} is satisfied. Our result now follows from Theorem \ref{thm:meas1}. \qedhere	
\end{proof}

Observe that the negation of \eqref{eq:cond2} from Theorem~\ref{thm:meas1} is
$$
\limsup_{n\to\infty}\#\psi^{-1}(n)e^{-h_\mu n+\sqrt{2\rho_\mu n}g(n) }<\infty.
$$
Unfortunately, this condition does not correspond to the convergent series condition given in Theorem~\ref{thm:convergence}. Therefore we cannot use this condition to obtain a simple criterion for determining the measure of $R_{\psi}$.

After examining Theorem \ref{thm:convergence} and recalling the analogy made earlier with works on shrinking targets where formulas of the form \eqref{volume sums} play an important role, a natural candidate for a divergence criterion for ensuring $R_{\psi}$ will have full measure is
$$
\sum_{n=1}^\infty e^{-h_\mu\psi(n)+\sqrt{2\rho_\mu\psi(n)\log\log\psi(n)}}=\infty.
$$
The next theorem demonstrates that this intuitive guess is incorrect and that this divergence criterion alone is insufficient for determining the measure of $R_{\psi}$.

\begin{theorem}\label{thm:nocharac}
	Let $\mu$ be a non-uniform Bernoulli measure and let $g\colon\N\mapsto[1,\infty)$ be a function satisfying $\limsup\limits_{n\to\infty}\frac{g(n)}{\sqrt{\log\log(n)}}\leq 1$ and $\lim\limits_{n\to\infty}g(n)=\infty$. Then there exists a monotone increasing function $\psi\colon\N\mapsto\N$ such that $\psi(n)\leq n$ for all $n\in\N$ and
	$$
	\sum_{n=1}^\infty e^{-h_\mu\psi(n)+\sqrt{2\rho_\mu\psi(n)}g(\psi(n))}=\infty
	$$
	but $\mu(R_\psi)=0$.
\end{theorem}
If we take $g(n)=\sqrt{\log\log{n}}$ in Theorem \ref{thm:nocharac}, we see that it is in conflict with the prediction made just before the statement of this theorem and thus this prediction cannot be correct.

What remains of the paper is organised as follows. In Section \ref{Sec:Tools} we recall some useful results from Probability Theory. In Sections \ref{Sec:Convergence result} and \ref{Sec:Fullmeasure} we prove Theorems \ref{thm:convergence} and \ref{thm:meas1} respectively. In Section \ref{Sec:Counterexample} we prove Theorem \ref{thm:nocharac}. Last of all, in Section \ref{Sec:self-similar} we apply our results to the study of dynamics on self-similar sets.

\section{Probabilistic tools}
\label{Sec:Tools}
In this section we recall some probabilistic results relating to Gibbs measures. The following two statements are the Central Limit Theorem and the Law of the Iterated Logarithm for Gibbs measures. These statements follow from \cite[Corollary~1 and Corollary~2]{DenPhi} and \eqref{eq:Gibbs}. We also refer the reader to \cite[Proposition 2.1 and Proposition 2.2]{JorPol} where these results are stated in our language.

\begin{theorem}[Central Limit Theorem]
	\label{thm:CLT}
Let $\mu$ be the Gibbs measure for a H\"older continuous potential that is not cohomologous to a constant. Then we have
$$\lim_{n\to\infty}\mu\left(\left\{\iii\in \Sigma_{A}:\frac{\log\mu([\iii|_1^n])+h_\mu n}{\sqrt{n}}\leq t\right\}\right )	=\frac{1}{\sqrt{2\pi\rho_{\mu}^2}}\int_{-\infty}^{t}e^{-x^2/2\rho_{\mu}^{2}}\, dx$$
\end{theorem}

\begin{theorem}[Law of the Iterated Logarithm] \label{thm:LIL} Let $\mu$ be the Gibbs measure for a H\"older continuous potential that is not cohomologous to a constant. Then we have
	$$
	\limsup_{n\to\infty}\frac{\log\mu([\iii|_1^n])+h_\mu n}{\sqrt{2\rho_\mu n\log\log{n}}}=1\quad \text{ for $\mu$-almost every $\iii  \in \Sigma_{A}$.}
	$$
\end{theorem}
We will also use the following well known property of Gibbs measures for H\"older continuous potentials which states that they have exponential decay of correlations. A proof of the following statement is contained within the proof of \cite[Proposition~1.14]{Bowen}.

\begin{theorem}[Exponential Decay of Correlations]\label{thm:decay}
	Let $\mu$ be the Gibbs measure for a H\"older continuous potential. Then there exist constants $D>0$ and $0<\gamma<1$ such that for every $\iii,\jjj\in\Sigma_{A,*}$ and $n\geq|\iii|$ we have
	$$
	\left|\mu([\iii]\cap\sigma^{-n}[\jjj])-\mu([\iii])\mu([\jjj])\right|\leq D\gamma^{n-|\iii|}\mu([\iii])\mu([\jjj]).
	$$
\end{theorem}

In the special case of Bernoulli measures, we will make use of the following more sophisticated estimate in the Central Limit Theorem. A proof of this statement can be found in \cite[Theorem~9.4]{Bill}.

\begin{theorem}\label{thm:CLT2}
	Let $\mu$ be a non-uniform Bernoulli measure and let $(a_n)_{n=1}^{\infty}$ be a sequence of real numbers such that $\lim_{n\to\infty}a_n=\infty$ and $\lim_{n\to\infty}\frac{a_n}{\sqrt{n}}=0$. Then
	$$
	\mu\left(\left\{\iii\in\Sigma:\frac{\log\mu([\iii|_1^n])+h_\mu n}{\sqrt{\rho_\mu}}\geq a_n\sqrt{n} \right\}\right)=e^{-a_n^2(1+\zeta_n)/2},
	$$
	for some sequence of real numbers $(\zeta_n)_{n=1}^{\infty}$ satisfying  $\lim_{n\to\infty}\zeta_n =0$.
\end{theorem}

Finally, from time-to-time, we will also call upon the first Borel--Cantelli Lemma, see for example \cite[Lemma 1.2]{Harman}.

\begin{lemma} \label{lem:BC}
Let $(X,\mathcal{B},\mu)$ be a measure space. Let $(A_j)_{j=1}^{\infty}$ be a collection of measurable subsets of $X$. If 
\[\sum_{j=1}^{\infty}{\mu(A_j)} < \infty,\]
then 
\[\mu\left(\{x \in X: x \in A_j \text{ for infinitely many } j \in \N\} \right)=0.\]
\end{lemma}

\section{Proof of Theorem \ref{thm:convergence}}
\label{Sec:Convergence result}
In this section we prove Theorem \ref{thm:convergence}.

\begin{proof}
Let $\psi\colon\N\mapsto\N$ be a function satisfying
\begin{equation}
\label{psi convergence}
\sum_{n=1}^\infty e^{-h_\mu\psi(n)+(1+\varepsilon)\sqrt{2\rho_\mu\psi(n)\log\log\psi(n)}}<\infty
\end{equation} for some $\varepsilon>0$. Define $\tilde{\psi}:\mathbb{N}\to\mathbb{N}$ according to the rule
\[ \tilde{\psi}(n) = \left\{ \begin{array}{ll}
\psi(n) & \mbox{if $\psi(n)\leq n$},\\[2ex]
n & \mbox{if $\psi(n)>n$}.\end{array} \right. \]
Notice that $\tilde{\psi}(n)\leq \psi(n)$ for all $n\in\mathbb{N}$. This implies $R_{\psi}\subset R_{\tilde{\psi}}.$ We also have
\begin{align*}
&\sum_{n=1}^\infty e^{-h_\mu\tilde{\psi}(n)+(1+\varepsilon)\sqrt{2\rho_\mu\tilde{\psi}(n)\log\log\tilde{\psi}(n)}}\\
\leq &\sum_{n=1}^\infty e^{-h_\mu\psi(n)+(1+\varepsilon)\sqrt{2\rho_\mu\psi(n)\log\log\psi(n)}}+\sum_{n=1}^\infty e^{-h_\mu n+(1+\varepsilon)\sqrt{2\rho_\mu n\log\log n}}\\
<&\infty.
\end{align*}Therefore $\tilde{\psi}$ also satisfies \eqref{psi convergence}. This means that without loss of generality we may assume that our original $\psi$ satisfies $\psi(n)\leq n$ for all $n\in \N$. The rest of our proof is given under this assumption.

By Theorem~\ref{thm:LIL}, $\mu\left(\bigcup_{N=1}^\infty E_N\right)=1$, where
$$
E_N=\left\{\iii\in\Sigma_{A}:\mu([\iii|_1^n])\leq e^{-h_\mu n+(1+\varepsilon)\sqrt{2\rho_\mu n\log\log{n}}}\text{ for all }n\geq N\right\}.
$$
Thus, to prove $\mu(R_{\psi})=0$ it is enough to show that $\mu(E_N\cap R_\psi)=0$ for every $N\in\N$. This we do below. In what follows $N\in \N$ is fixed.

For each $n\in\mathbb{N}$ let
$$
F_{A,n}=\{\iii\in\Sigma_{A,n}:\mu([\iii])\leq e^{-h_\mu n+(1+\varepsilon)\sqrt{2\rho_\mu n\log\log{n}}}\}.
$$
We also let $$M(N):=\max\{k\geq1:\psi(k)< N\}+1.$$

By the expression for $R_\psi$ given in \eqref{eq:simpleform}, we know that {$$E_N\cap R_{\psi}=\bigcap_{n=1}^\infty\bigcup_{k=n}^\infty  \bigcup_{\substack{\iii\in\Sigma_{A,\psi(k)}\\\jjj\in\Sigma_{A,k-\psi(k)}\\ \iii\jjj\iii\in \Sigma_{A,k+\psi(k)}}}E_N \cap[\iii\jjj\iii].$$} Therefore to prove $\mu(R_\psi)=0$, by the first Borel--Cantelli Lemma (Lemma \ref{lem:BC}) it is sufficient to show that
\begin{equation}
\label{Eq:Needtoshow}\sum_{k=1}^{\infty}\sum_{\substack{\iii\in\Sigma_{A,\psi(k)}\\\jjj\in\Sigma_{A,k-\psi(k)}\\ \iii\jjj\iii\in \Sigma_{A,k+\psi(k)}}}\mu(E_N\cap[\iii\jjj\iii])<\infty.
\end{equation} This we do presently. We begin by observing that
\begin{align*}
\sum_{k=1}^\infty\sum_{\substack{\iii\in\Sigma_{A,\psi(k)}\\\jjj\in\Sigma_{A,k-\psi(k)}\\ \iii\jjj\iii\in \Sigma_{A,k+\psi(k)}}}\mu(E_N\cap[\iii\jjj\iii])&\leq M(N)+\sum_{k=M(N)+1}^\infty\sum_{\substack{\iii\in\Sigma_{A,\psi(k)}\\\jjj\in\Sigma_{A,k-\psi(k)}\\ \iii\jjj\iii\in \Sigma_{A,k+\psi(k)}}}\mu(E_N\cap[\iii\jjj\iii])\\
	&\leq M(N)+\sum_{k=M(N)+1}^\infty\sum_{\substack{\iii\in F_{A,\psi(k)}\\\jjj\in\Sigma_{A,k-\psi(k)}\\ \iii\jjj\iii\in \Sigma_{A,k+\psi(k)}}}\mu([\iii\jjj\iii]).
\end{align*}
Using \eqref{eq:quasiBernoulli} and the definition of $F_{A,\psi(k)},$ it follows that	
\begin{align*}
&\sum_{k=1}^\infty\sum_{\substack{\iii\in\Sigma_{A,\psi(k)}\\\jjj\in\Sigma_{A,k-\psi(k)}\\ \iii\jjj\iii\in \Sigma_{A,k+\psi(k)}}}\mu(E_N\cap[\iii\jjj\iii])\\
\leq& M(N)+C\sum_{k=M(N)+1}^\infty\sum_{\substack{\iii\in F_{A,\psi(k)}\\\jjj\in\Sigma_{A,k-\psi(k)}\\ \iii\jjj\in \Sigma_{A,k}}}\mu([\iii\jjj])e^{-h_\mu\psi(k)+(1+\varepsilon)\sqrt{2\rho_\mu\psi(k)\log\log \psi(k)}}\\
	\leq& M(N)+C\sum_{k=M(N)+1}^\infty e^{-h_\mu\psi(k)+(1+\varepsilon)\sqrt{2\rho_\mu\psi(k)\log\log \psi(k)}}\\
	<&\infty.
\end{align*}
In the final line we have used \eqref{psi convergence}. Thus we have shown that \eqref{Eq:Needtoshow} holds and this completes our proof.
\end{proof}
\section{Proof of Theorem \ref{thm:meas1}}
\label{Sec:Fullmeasure}

In this section we prove Theorem \ref{thm:meas1}. We now fix $\mu$, $\psi$, $(n_k)_{k=1}^{\infty}$, and $g$ so that the assumptions of this theorem are satisfied. Let
$$S_{\psi,(n_k)}:=\left\{\iii\in \Sigma_{A}:d(\sigma^{p}\iii,\iii)\leq K^{-n_k}\textrm{ for some }p\in \psi^{-1}(n_k)\textrm{ for infinitely many }k\in\mathbb{N}\right\}.$$
 By changing the values that $\psi$ takes for at most finitely many natural numbers we may ensure that $\psi(n)\leq n$ for all $n\in \N$. This assumption does not change the set $S_{\psi,(n_k)}$ or any of our underlying assumptions. As such, without loss of generality we can give our proof under this assumption.

To prove that $\mu(S_{\psi,(n_k)})=1,$ we will show that there exists a constant $c>0$ such that for any $\iii\in \Sigma_{A,*}$ we have
\begin{equation}
\label{eq:density}
\mu([\iii]\cap S_{\psi,(n_k)})\geq c\mu([\iii]).
\end{equation} It then follows from \cite[Lemma 6]{BDV} that $\mu(S_{\psi,(n_k)})=1$. We can think of \cite[Lemma~6]{BDV} as essentially providing us with an analogue of the classical Lebesgue Density Theorem (see, for example, \cite[Corollary 2.14]{Mattila}) that holds for more general measures.
Since $S_{\psi,(n_k)}\subset R_{\psi}$ we also immediately have that $\mu(R_{\psi})=1$. Therefore to complete the proof of Theorem \ref{thm:meas1} it remains to show that \eqref{eq:density} holds. We split our proof of \eqref{eq:density} into three parts below.\\

\noindent \textbf{Part 1. Construction of the auxiliary sets.}

It follows from \eqref{eq:quasiBernoulli} that the set $$\left\{\iii\in\Sigma_{A}:\limsup\limits_{k\to\infty}\frac{\log\mu([\iii|_1^{n_k}])+h_\mu n_k}{\sqrt{2\rho_\mu n_k}g(n_k)}> 1+\varepsilon\right\}$$ is $\sigma$-invariant for every $\varepsilon>0.$ Using the fact that $\mu$ is ergodic together with \eqref{eq:cond1}, we may conclude that there exists $\varepsilon>0$ such that
\begin{equation}\label{eq:cond1b}
\mu\left(\left\{\iii\in\Sigma_{A}:\limsup_{k\to\infty}\frac{\log\mu([\iii|_1^{n_k}])+h_\mu n_k}{\sqrt{2\rho_\mu n_k}g(n_k)}> 1+\varepsilon\right\}\right)=1.
\end{equation}
For the rest of our proof $\varepsilon>0$ is fixed so that \eqref{eq:cond1b} is satisfied.

Without loss of generality, we may assume that our sequence $(n_k)_{k=1}^{\infty}$ is such that
\begin{equation}\label{eq:bound5}
\frac{e^{\varepsilon\sqrt{2n_k\rho_\mu}g(n_k)/4}}{1+n_k+\lceil\frac{-h_\mu}{\log\gamma}\rceil \cdot n_k}\geq D+1
\end{equation}
for every $k\in\N$. Here and throughout our proof, $\gamma\in(0,1)$ and $D>0$ are as in Theorem~\ref{thm:decay}.

Given $\delta>0$ we let $a_1(\delta)\in\mathbb{N}$ be such that for every $n_k\geq a_1(\delta)$ we have
\begin{equation}\label{eq:bound4}
\exp\left(-\#\psi^{-1}(n_k)e^{-h_\mu n_k+\sqrt{2\rho_\mu n_k}g(n_k)}\right)\leq\delta.
\end{equation} The fact $a_1(\delta)$ always exists is a consequence of \eqref{eq:cond2}.

For simplicity, we let 
$$F:=\min_{\iii\in\Sigma_{A}}f(\iii)-P(f).$$ For $\delta>0$ and $m\in\N,$ let $a(m,\delta)>\max\{m,a_1(\delta)\}$ be such that for every $n\geq a(m,\delta)$ we have
\begin{equation}\label{eq:bound12}
\left|\frac{-2\log C+m(F+h_{\mu})}{\sqrt{2\rho_\mu n}g(n)}\right|<\varepsilon/4 \quad \text{and}\quad \frac{\sqrt{n-m}g(n-m)}{\sqrt{n}g(n)}\geq\frac{1+\varepsilon/2}{1+3\varepsilon/4}.
\end{equation}
The latter inequality is possible because of our assumption that $\lim_{n\to\infty}\frac{g(n)}{g(n+1)}=1$. Here $C$ is the constant appearing in \eqref{eq:Gibbs} and \eqref{eq:quasiBernoulli}. The equations \eqref{eq:Gibbs}, \eqref{eq:quasiBernoulli} and \eqref{eq:bound12} imply the following estimate. 

\begin{lemma} \label{lem:bound3} For every $\iii\in\Sigma_{A,*}$, if $\jjj\in \Sigma_{A,*}$ satisfies $|\jjj|> a(|\iii|,\delta),$ $\iii\jjj\in \Sigma_{A,*},$ and
\[\frac{\log\mu([\jjj])+h_\mu|\jjj|}{\sqrt{2\rho_\mu |\jjj|}g(|\jjj|)}\geq1+\frac{3\varepsilon}{4},\]
then
\[\frac{\log\mu([\iii\jjj])+h_\mu|\iii\jjj|}{\sqrt{2\rho_\mu |\iii\jjj|}g(|\iii\jjj|)}\geq1+\frac{\varepsilon}{4}.\]
\end{lemma}

\begin{proof}
First, we note that it follows from \eqref{eq:quasiBernoulli} that
\begin{align*}
\frac{\log\mu([\iii\jjj])+h_\mu|\iii\jjj|}{\sqrt{2\rho_\mu |\iii\jjj|}g(|\iii\jjj|)} &\geq \frac{-\log{C}+\log{\mu([\iii])+h_{\mu}|\iii|}}{\sqrt{2\rho_\mu |\iii\jjj|}g(|\iii\jjj|)} + \frac{\log{\mu([\jjj])}+h_{\mu}|\jjj|}{\sqrt{2\rho_\mu |\iii\jjj|}g(|\iii\jjj|)}.
\end{align*}
It then follows from \eqref{eq:Gibbs} that
\begin{align*}
\frac{\log\mu([\iii\jjj])+h_\mu|\iii\jjj|}{\sqrt{2\rho_\mu |\iii\jjj|}g(|\iii\jjj|)}
           &\geq \frac{-2\log{C}+\sum_{k=0}^{|\iii|-1}{f(\sigma^k\iii)}-|\iii|P(f)+h_{\mu}|\iii|}{\sqrt{2\rho_\mu |\iii\jjj|}g(|\iii\jjj|)} + \frac{\log{\mu([\jjj])}+h_{\mu}|\jjj|}{\sqrt{2\rho_\mu |\iii\jjj|}g(|\iii\jjj|)} \\
           &\geq \frac{-2\log{C}+|\iii|(F+h_{\mu})}{\sqrt{2\rho_\mu |\iii\jjj|}g(|\iii\jjj|)} + \frac{\log{\mu([\jjj])}+h_{\mu}|\jjj|}{\sqrt{2\rho_\mu |\iii\jjj|}g(|\iii\jjj|)}.
\end{align*}
Next we rewrite the fraction on the far right-hand side and then apply the second inequality of ~\eqref{eq:bound12} and the assumption of the lemma to obtain
\begin{align*}
\frac{\log\mu([\iii\jjj])+h_\mu|\iii\jjj|}{\sqrt{2\rho_\mu |\iii\jjj|}g(|\iii\jjj|)}
           &\geq \frac{-2\log{C}+|\iii|(F+h_{\mu})}{\sqrt{2\rho_\mu |\iii\jjj|}g(|\iii\jjj|)} + \left(\frac{\log{\mu([\jjj])}+h_{\mu}|\jjj|}{\sqrt{2\rho_{\mu}|\jjj|}g(|\jjj|)}\right)\left(\frac{\sqrt{2\rho_{\mu}|\jjj|}g(|\jjj|)}{\sqrt{2\rho_\mu |\iii\jjj|}g(|\iii\jjj|)}\right) \\
           &\geq \frac{-2\log{C}+|\iii|(F+h_{\mu})}{\sqrt{2\rho_\mu |\iii\jjj|}g(|\iii\jjj|)} + \left(1+\frac{\varepsilon}{2}\right).
\end{align*}
Finally, the claimed bound is now obtained by applying the first inequality of \eqref{eq:bound12}.
\end{proof}

For every $\delta>0$, and $m\in\N,$ we define $b(m,\delta)$ to be a sufficiently large natural number for which the set
\begin{align*}
G_{h,m,\delta}:=\Bigg\{\jjj\in [h]:&\text{ there exists $n_k\in[a(m,\delta),b(m,\delta)]$ such that }\\
&\phantom{==========}\frac{\log\mu([\jjj|_1^{n_k - m}])+h_\mu (n_k-m)}{\sqrt{2\rho_\mu (n_k-m)}g(n_k-m)}\geq1+\frac{3\varepsilon}{4}\Bigg\}
\end{align*}
satisfies
\begin{equation}\label{eq:boundB}
\mu\left(G_{h,m,\delta}\right)\geq(1-\delta)\mu([h])
\end{equation}
for all $h\in\{1,\ldots,K\}$. The existence of $b(m,\delta)$ follows from \eqref{eq:cond1b} and an application of \eqref{eq:Gibbs}, \eqref{eq:quasiBernoulli}, and our assumptions on $g$. By definition we have
\begin{equation}\label{eq:disjoint}
G_{h,m,\delta}\cap G_{h',m,\delta}=\emptyset\text{ for $h\neq h'$.}
\end{equation}

We can partition each $G_{h,m,\delta}$ into a disjoint collection of cylinders as follows. For each $\ell\in\mathbb{N}$ such that $n_{\ell}\in [a(m,\delta),b(m,\delta)]$ let
\begin{equation}\label{eq:defG}
\begin{split}
G_{h,m,\delta}^{(\ell)}=&\bigg\{\jjj\in \Sigma_{A,n_\ell-m}:j_1=h,\quad \frac{\log\mu([\jjj|_1^{n_\ell-m}])+h_\mu(n_\ell-m)}{\sqrt{2\rho_\mu(n_\ell-m)}g(n_\ell-m)}\geq1+\frac{3\varepsilon}{4}\\
&\qquad\text{ but }\frac{\log\mu([\jjj|_1^{n_k-m}])+h_\mu(n_k-m)}{\sqrt{2\rho_\mu(n_k-m)}g(n_k-m)}<1+\frac{3\varepsilon}{4}\text{ for all }n_k\in[a(m,\delta),n_\ell-1]\bigg\}.
\end{split}
\end{equation}
Then
 $$G_{h,m,\delta}=\bigcup_{\ell:n_\ell\in[a(m,\delta),b(m,\delta)]}[G_{h,m,\delta}^{(\ell)}]\quad \textrm{ and }\quad  [G_{h,m,\delta}^{(\ell_1)}]\cap[G_{h,m,\delta}^{(\ell_2)}]=\emptyset$$ for $\ell_1\neq\ell_2$.

Now we define two sets of words whose corresponding elements of $\Sigma_{A}$ either satisfy or fail a recurrence property defined in terms of a cylinder of length $n_\ell$. Given $m\in \N$ and $\delta>0$ let
$$c(m,\delta):=\max\left\{\bigcup_{k=1}^{b(m,\delta)}{\psi^{-1}(k)}\right\}+b(m,\delta).$$ { This significance of the quantity $c(m,\delta)$ is seen as follows. Given $\iii\in \Sigma_{A,*}$ and $\delta>0,$ using the fact that $\psi(n)\leq n$ for all $n\in N$ and the definition of $c(|\iii|,\delta)$, it follows that for any $n_{\ell}\in [a(|\iii|,\delta),b(|\iii|,\delta)]$ we have
	\begin{equation}
	\label{eq:psiinclusion}
	\psi^{-1}(n_\ell) \subset[n_\ell,c(|\iii|,\delta)-n_\ell].
	\end{equation} This means that if we are interested in sequences satisfying $d(\sigma^{p}\jjj,\jjj)\leq K^{-n_\ell}$ for some $n_{\ell}\in [a(|\iii|,\delta),b(|\iii|,\delta)]$ and $p\in \psi^{-1}(n_\ell),$ then it is sufficient to know only the first $c(|\iii|,\delta)$ entries in~$\jjj$. This fact underpins the definitions of $B_{\iii,\delta}^{(\ell)}$ and $C_{\iii,\delta}^{(\ell)}$ below.}

\vbox{For a word $\iii\in\Sigma_{A,*},$ $\delta>0$ and $l\in\mathbb{N}$ such that $n_\ell\in [a(|\iii|,\delta),b(|\iii|,\delta)],$ let
\begin{align*}
B_{\iii,\delta}^{(\ell)}:=\bigg\{\jjj\in\Sigma_{A,c(|\iii|,\delta)-|\iii|}:&\jjj|_1^{n_\ell-|\iii|}\in \bigcup_{h:\iii h \in \Sigma_{A,*}}G_{h,|\iii|,\delta}^{(\ell)}\text{ and }\jjj|_{p-|\iii|+1}^{p-|\iii|+n_\ell}=(\iii\jjj)|_1^{n_\ell}\\
&\phantom{==================}\text{ for some }p\in\psi^{-1}(n_\ell)\bigg\}
\end{align*}}
and
\begin{align*}
C_{\iii,\delta}^{(\ell)}:=\bigg\{\jjj\in\Sigma_{A,c(|\iii|,\delta)-|\iii|}:&\jjj|_1^{n_\ell-|\iii|}\in \bigcup_{h:\iii h \in \Sigma_{A,*}}G_{h,|\iii|,\delta}^{(\ell)}\text{ and }\jjj|_{p-|\iii|+1}^{p-|\iii|+n_\ell}\neq(\iii\jjj)|_1^{n_\ell}\\
&\phantom{==================}\text{ for every }p\in\psi^{-1}(n_\ell)\bigg\}.
\end{align*}
We also let $$B_{\iii,\delta}:=\bigcup_{\ell:n_\ell\in[a(|\iii|,\delta),b(|\iii|,\delta)]}B_{\iii,\delta}^{(\ell)}.$$ \label{B definition}
An important consequence of the definition of $B_{\iii,\delta}^{(\ell)}$ is that if $\jjj'\in \iii B_{\iii,\delta}^{(\ell)}$ then
\begin{equation}
\label{Recurrence consequence}
d(\sigma^{p}\jjj',\jjj')\leq K^{-n_\ell}
\end{equation} for some $p\in \psi^{-1}(n_\ell)$.\\

\noindent \textbf{Part 2. Measure properties of our auxiliary sets.}

\begin{lemma}\label{lem:Ci} For every $\iii\in\Sigma_{A,*}$, $\delta>0,$ and $\ell \in \N$ such that $n_\ell\in [a(|\iii|,\delta),b(|\iii|,\delta)],$ we have
	\[
	\mu([\iii C_{\iii,\delta}^{(\ell)}])\leq C\delta \sum_{h:\iii h \in \Sigma_{A,*}}\mu([\iii G_{h,|\iii|,\delta}^{(\ell)}]).
	\]
\end{lemma}

\begin{proof}
	Let $r_\ell= \lceil \frac{-h_\mu}{\log\gamma}\rceil \cdot  n_\ell$. Let $S$ denote the $(n_\ell+r_\ell)$-separated subset of $\psi^{-1}(n_\ell)$ defined inductively as follows. Let $S_1=\{\min\psi^{-1}(n_\ell)\}$ and for $k\geq 2$ let \[S_k=S_{k-1}\cup\left\{\min\psi^{-1}(n_\ell)\cap[\max S_{k-1}+n_\ell+r_\ell+1,\infty)\right\}.\] We then let $S=\bigcup_{k}S_k$. Clearly,
	\begin{equation}\label{eq:boundS}
	\frac{\#\psi^{-1}(n_\ell)}{n_\ell+r_\ell+1}\leq\#S\leq\#\psi^{-1}(n_\ell).
	\end{equation}
	By \eqref{eq:psiinclusion} we know that $S\subset[n_\ell,c(\iii,\delta)-n_\ell]$. It follows now from the definition of $C_{\iii, \delta}^{(\ell)}$ and \eqref{eq:disjoint} that
\[
    \begin{split}
	&\mu([\iii C_{\iii,\delta}^{(\ell)}])\\
	\leq &\mu\left(\left\{\jjj\in\Sigma_{A}:\jjj|_1^{|\iii|}=\iii,\ \jjj|_{|\iii|+1}^{n_\ell}\in \bigcup_{h:\iii h \in \Sigma_{A,*}}G_{h,|\iii|,\delta}^{(\ell)}\text{ and }\jjj|_{p+1}^{p+n_\ell}\neq\jjj|_1^{n_\ell}\text{ for every }p\in S\right\}\right)\\
=&\sum_{h:\iii h \in \Sigma_{A,*}}\sum_{\jjj\in G_{h,|\iii|,\delta}^{(\ell)}}\mu\left([\iii \jjj]\cap\left(\bigcap_{p\in S}\sigma^{-p}[\iii \jjj]^c\right)\right).
\end{split}
\]
	Combining the above with Theorem~\ref{thm:decay} and \eqref{eq:quasiBernoulli}, we have 
	$$
\mu([\iii C_{\iii,\delta}^{(\ell)}])\leq C\sum_{h:\iii h \in \Sigma_{A,*}}\sum_{\jjj\in G_{h,|\iii|,\delta}^{(\ell)}}\mu([\iii \jjj])\left(1+D\gamma^{r_\ell}\right)^{\#S}\left(1-\mu([\iii\jjj])\right)^{\#S}.
	$$
Now using the fact that $1+x\leq e^x$ for all $x\in \mathbb{R}$ together with \eqref{eq:boundS}, it follows that
\begin{align*}
\mu([\iii C_{\iii,\delta}^{(\ell)}])
     &\leq C\sum_{h:\iii h \in \Sigma_{A,*}}\sum_{\jjj\in G_{h,|\iii|,\delta}^{(\ell)}}{\mu([\iii\jjj])\exp\left(\#S D\gamma^{r_\ell}\right)\exp\left(-\#S\mu([\iii\jjj])\right)} \\
     &\leq C\sum_{h:\iii h \in \Sigma_{A,*}}\sum_{\jjj\in G_{h,|\iii|,\delta}^{(\ell)}}{\mu([\iii\jjj])\exp\left(D\gamma^{r_\ell} \#\psi^{-1}(n_{\ell})-\mu([\iii\jjj]) \frac{\#\psi^{-1}(n_{\ell})}{n_{\ell}+r_{\ell}+1}\right).}
\end{align*}
By definition, see \eqref{eq:defG}, any word $\jjj \in G_{h,|\iii|,\delta}^{(\ell)}$ has $|\jjj| = n_{\ell}-|\iii|$. Furthermore, it also follows directly from the definition of $G_{h,|\iii|,\delta}^{(\ell)}$ that the hypothesis of Lemma \ref{lem:bound3} holds, and so we have
\[\mu([\iii\jjj])\geq\exp\left(-h_{\mu}|\iii\jjj| + \sqrt{2\rho_{\mu}|\iii\jjj|}g(|\iii\jjj|)(1+\varepsilon/4)\right).\] 
Putting this all together, we conclude that
\begin{align*}
&\mu([\iii C_{\iii,\delta}^{(\ell)}])\\
     \leq &C\sum_{h:\iii h \in \Sigma_{A,*}}\sum_{\jjj\in G_{h,|\iii|,\delta}^{(\ell)}}{\mu([\iii\jjj])\exp\left(D\gamma^{r_\ell} \#\psi^{-1}(n_{\ell})- \frac{\#\psi^{-1}(n_{\ell})}{n_{\ell}+r_{\ell}+1} e^{-h_{\mu}|\iii\jjj| + \sqrt{2\rho_{\mu}|\iii\jjj|}g(|\iii\jjj|)(1+\varepsilon/4)}\right)} \\
     \leq &C\sum_{h:\iii h \in \Sigma_{A,*}}{\mu([\iii G_{h,|\iii|,\delta}^{(\ell)}])\exp\left(D\gamma^{r_\ell} \#\psi^{-1}(n_{\ell})- \frac{\#\psi^{-1}(n_{\ell})}{n_{\ell}+r_{\ell}+1}  e^{-h_{\mu}n_{\ell} + \sqrt{2\rho_{\mu}n_{\ell}}g(n_{\ell})(1+\varepsilon/4)}\right)}.
\end{align*}
By equation \eqref{eq:bound5} and the definition of $r_\ell$ we have
	$$
	e^{-h_{\mu}n_\ell+\sqrt{2\rho_\mu n_\ell}g(n_\ell)}\left(\frac{e^{\varepsilon\sqrt{2\rho_\mu n_\ell}g(n_\ell)/4}}{n_\ell+r_\ell+1}-1\right)\geq D\gamma^{r_\ell}.
	$$
	Thus, it follows that
	\begin{align*}
	\mu([\iii C_{\iii,\delta}^{(\ell)}])&\leq C\sum_{h:\iii h \in \Sigma_{A,*}}{\mu([\iii G_{h,|\iii|,\delta}^{(\ell)}])\exp\left(-\#\psi^{-1}(n_{\ell})e^{-h_{\mu}n_{\ell}+\sqrt{2\rho_{\mu}n_{\ell}}g(n_{\ell})}\right)}.
	\end{align*}
	Finally, from the definition of $a(|\iii|,\delta)$ and \eqref{eq:bound4}, we see that
	\begin{align*}
	\mu([\iii C_{\iii,\delta}^{(\ell)}])&\leq C\delta\sum_{h:\iii h \in \Sigma_{A,*}}\mu([\iii G_{h,|\iii|,\delta}^{(\ell)}]),
	\end{align*}
which completes our proof.
\end{proof} 

\noindent \textbf{Part 3. Proof of \eqref{eq:density}.}

Let us fix now an arbitrary $\iii\in\Sigma_{A,*}$ and set about proving \eqref{eq:density}. Recall that we defined
$$S_{\psi,(n_k)}:=\left\{\iii\in \Sigma_{A}:d(\sigma^{p}\iii,\iii)\leq K^{-n_k}\textrm{ for some }p\in \psi^{-1}(n_k)\textrm{ for infinitely many }k\in\mathbb{N}\right\}.$$
We now define a subset of $S_{\psi,(n_k)}$ by induction as follows. First, let $k_1$ be such that $D\gamma^{k_1}<1$ and let
$$
D_1(\iii):=\bigcup_{\substack{\kkk\in\Sigma_{A,1+k_1} \\ \iii\kkk\in \Sigma_{A,*}} }[\iii\kkk B_{\iii\kkk,2^{-1}/C}].
$$ By the definition of $D_{1}(\iii)$, and noting that each $B_{\iii\kkk,2^{-1}/C}$ is a finite union of cylinders, we may choose a finite set of words $W_{1}(\iii)\subset \Sigma_{A,*}$ such that $$D_{1}(\iii)=\bigcup_{\jjj\in W_{1}(\iii)}[\jjj]\text{ and } [\jjj_{1}]\cap [\jjj_{2}]=\emptyset \text{ for }\jjj_1\neq \jjj_2.$$ Now suppose $D_{k}(\iii)\subset \Sigma_{A}$ and $W_{k}(\iii)\subset \Sigma_{A,*}$ are defined for some $k\geq 1$. We then define $$D_{k+1}(\iii)=\bigcup_{\jjj\in W_{k}(\iii)}\bigcup_{\substack{\kkk \in \Sigma_{A,k+1+k_1} \\ \jjj\kkk\in \Sigma_{A,*}}}[\jjj\kkk B_{\jjj\kkk,2^{-k-1}/C}].$$ 
We also let $W_{k+1}(\iii) \subset \Sigma_{A,*}$ be a finite set of words such that $$D_{k+1}(\iii)=\bigcup_{\jjj\in W_{k+1}(\iii)}[\jjj]\text{ and } [\jjj_{1}]\cap [\jjj_{2}]=\emptyset \text{ for }\jjj_1\neq \jjj_2.$$ Proceeding inductively we define $D_{k}(\iii)$ and $W_{k}(\iii)$ for all $k \in \N$. We observe that
\begin{equation}
\label{eq:inclusion}
D_{k+1}(\iii)\subset D_{k}(\iii)
\end{equation}
for all $k\geq 1,$ and
\begin{equation}
\label{eq:intersectioninclusion}
\bigcap_{k=1}^{\infty}D_{k}(\iii)\subset [\iii]\cap S_{\psi,(n_k)}.
\end{equation} This last inclusion follows from \eqref{Recurrence consequence}. Therefore to prove \eqref{eq:density} it is sufficient to obtain lower bounds for the measure of $\cap_{k=1}^{\infty}D_{k}(\iii).$ The following lemma allows us to do this and bounds how much measure is lost as we pass from $D_{k}(\iii)$ to $D_{k+1}(\iii)$.

\begin{lemma}\label{lem:boundD}
	For every $k\geq2$ we have
	$$
	\mu(D_k(\iii))\geq\mu(D_{k-1}(\iii))(1-D\gamma^{k+k_1})(1-2^{-k})(1-2^{-k}/C),
	$$
	and in the case when $k=1$ we have $$\mu(D_1(\iii))\geq \mu([\iii])(1-D\gamma^{k_1})(1-2^{-1})(1-2^{-1}/C).$$
\end{lemma}
\begin{proof}
\vbox{Let $k\geq 2$. By the definition of $D_k(\iii)$, and recalling the definitions of $B_{\iii,\delta}$, $B_{\iii,\delta}^{(\ell)}$, and $C_{\iii,\delta}^{(\ell)}$ as well as noting that for different values of $\ell$ the sets $G_{h,|\jjj|+k+k_1,2^{-k}/C}^{(\ell)}$ are disjoint by construction, we have
\begin{align*}
&\mu(D_k(\iii)) \\
&\phantom{==}=\sum_{\jjj\in W_{k-1}(\iii)}\sum_{\stackrel{\kkk\in\Sigma_{A,k+k_1}}{\jjj\kkk\in\Sigma_{A,*}}}\mu([\jjj\kkk B_{\jjj\kkk,2^{-k}/C}])\\
&\phantom{==}=\sum_{\jjj\in W_{k-1}(\iii)}\sum_{\stackrel{\kkk\in\Sigma_{A,k+k_1}}{\jjj\kkk\in\Sigma_{A,*}}}\sum_{\ell:n_\ell\in[a(|\jjj|+k+k_1,2^{-k}/C),b(|\jjj|+k+k_1, 2^{-k}/C)]}\mu([\jjj\kkk B_{\jjj\kkk,2^{-k}/C}^{(\ell)}])\\
&\phantom{==}=\sum_{\jjj\in W_{k-1}(\iii)}\sum_{\stackrel{\kkk\in\Sigma_{A,k+k_1}}{\jjj\kkk\in \Sigma_{A,*}}}\sum_{\ell:n_\ell\in[a(|\jjj|+k+k_1,2^{-k}/C),b(|\jjj|+k+k_1, 2^{-k}/C)]}\sum_{h:\jjj\kkk h\in \Sigma_{A,*}}\mu([\jjj\kkk G_{h,|\jjj|+k+k_1,2^{-k}/C}^{(\ell)}])\\
&\phantom{============}-\sum_{\jjj\in W_{k-1}(\iii)}\sum_{\stackrel{\kkk\in\Sigma_{A,k+k_1}}{\jjj\kkk\in \Sigma_{A,*}}}\sum_{\ell:n_\ell\in[a(|\jjj|+k+k_1,2^{-k}/C),b(|\jjj|+k+k_1, 2^{-k}/C)]}\mu([\jjj\kkk C_{\jjj\kkk,2^{-k}/C}^{(\ell)}]).
\end{align*}} 
Now using Lemma~\ref{lem:Ci} in the above we see that
\begin{align*}
&\mu(D_k(\iii))\\
&\geq \sum_{\jjj\in W_{k-1}(\iii)}\sum_{\stackrel{\kkk\in\Sigma_{A,k+k_1}}{\jjj\kkk\in \Sigma_{A,*}}}\sum_{\ell:n_\ell\in[a(|\jjj|+k+k_1,2^{-k}/C),b(|\jjj|+k+k_1, 2^{-k}/C)]}\sum_{h:\jjj\kkk h\in \Sigma_{A,*}}\mu([\jjj\kkk G_{h,|\jjj|+k+k_1,2^{-k}/C}^{(\ell)}])(1-2^{-k})\\
&=\sum_{\jjj\in W_{k-1}(\iii)}\sum_{\ell:n_\ell\in[a(|\jjj|+k+k_1,2^{-k}/C),b(|\jjj|+k+k_1, 2^{-k}/C)]}\sum_{h\in\{1,\ldots,K\}}\sum_{\stackrel{\kkk\in\Sigma_{A,k+k_1}}{\jjj\kkk h\in \Sigma_{A,*}}}\mu([\jjj\kkk G_{h,|\jjj|+k+k_1,2^{-k}/C}^{(\ell)}])(1-2^{-k}) \\
&=\sum_{\jjj\in W_{k-1}(\iii)}\sum_{\ell:n_\ell\in[a(|\jjj|+k+k_1,2^{-K}/C),b(|\jjj|+k+k_1, 2^{-K}/C)]}\sum_{h\in\{1,\ldots,K\}}\mu([\jjj]\cap \sigma^{-|\jjj|-k-k_1}[G_{h,|\jjj|+k+k_1,2^{-k}/C}^{(\ell)}])(1-2^{-k}).
\end{align*}
Next, by Theorem~\ref{thm:decay} we have
\begin{align*}
&\mu(D_{k}(\iii)) \\
&\geq \sum_{\jjj\in W_{k-1}(\iii)}\sum_{\ell: n_\ell\in[a(|\jjj|+k+k_1,2^{-k}/C),b(|\jjj|+k+k_1, 2^{-k}/C)]}\sum_{h\in\{1,\ldots,K\}}\mu([\jjj])\mu([G_{h,|\jjj|+k+k_1,2^{-k}/C}^{(\ell)}])(1-D\gamma^{k+k_1})(1-2^{-k})\\
&=\sum_{\jjj\in W_{k-1}(\iii)}\sum_{h\in\{1,\ldots,K\}}\mu([\jjj])\mu(G_{h,|\jjj|+k+k_1,2^{-k}/C})(1-D\gamma^{k+k_1})(1-2^{-k}).
\end{align*}
Now using \eqref{eq:boundB} in the above we see that
\begin{align*}
\mu(D_k(\iii))&\geq \sum_{\jjj\in W_{k-1}(\iii)}\sum_{h\in\{1,\ldots,K\}}\mu([\jjj])\mu([h])(1-D\gamma^{k+k_1})(1-2^{-k})(1-2^{-k}/C)\\
&=\sum_{\jjj\in W_{k-1}(\iii)}\mu([\jjj])(1-D\gamma^{k+k_1})(1-2^{-k})(1-2^{-k}/C)\\
&=\mu(D_{k-1}(\iii))(1-D\gamma^{k+k_1})(1-2^{-k})(1-2^{-k}/C).
\end{align*} The proof of the second claim is similar.
\end{proof}
It now follows from \eqref{eq:inclusion}, \eqref{eq:intersectioninclusion}, and Lemma \ref{lem:boundD} that 
\begin{align*}
\mu([i]\cap S_{\psi,(n_k)})&\geq \mu\left(\bigcap_{k=1}^{\infty}D_{k}(\iii)\right)\\
&=\lim_{k\to\infty}\mu(D_{k}(\iii))\\
&\geq \mu([\iii])\prod_{k=k_1}^\infty(1-D\gamma^k)\prod_{k=1}^\infty(1-2^{-k})\prod_{k=1}^\infty(1-2^{-k}/C).
\end{align*}Since the series $\sum_{k=k_1}^\infty D\gamma^k$ and $\sum_{k=1}^\infty2^{-k}$ are convergent we can take $$c=\prod_{k=k_1}^\infty(1-D\gamma^k)\prod_{k=1}^\infty(1-2^{-k})\prod_{k=1}^\infty(1-2^{-k}/C)$$ and \eqref{eq:density} holds (recall that $0 < \gamma < 1$ and that $k_1 \in \N$ was chosen such that $D\gamma^{k_1}<1$). This completes our proof.

\section{Proof of Theorem \ref{thm:nocharac}}
\label{Sec:Counterexample}
We begin our proof of Theorem \ref{thm:nocharac} by showing that for every function $g\colon\N\mapsto[1,\infty)$ satisfying $\limsup\limits_{n\to\infty}\frac{g(n)}{\sqrt{\log\log{n}}}\leq 1$ and $\lim\limits_{n\to\infty}g(n)=\infty,$ there exists a sequence $(n_k)_{k=1}^{\infty}$ such that for a typical $\iii\in \Sigma$ the quantity $\log\mu([\iii|_1^{n_k}])+h_\mu n_k$ eventually satisfies a useful upper bound formulated in terms of our function $g$.

\begin{lemma}\label{lem:sequence}
	Let $\mu$ be a non-uniform Bernoulli measure and let $g\colon\N\mapsto[1,\infty)$ be a function such that $\limsup\limits_{n\to\infty}\frac{g(n)}{\sqrt{\log\log{n}}}\leq 1$ and $\lim\limits_{n\to\infty}g(n)=\infty$. Then there exists a sequence of positive integers $(n_k)_{k=1}^{\infty}$ such that
	$$
	\mu\left(\left\{\iii\in\Sigma:\limsup_{k\to\infty}\frac{\log\mu([\iii|_1^{n_k}])+h_\mu n_k}{\sqrt{2\rho_\mu n_k}g(n_k)}\leq \frac{1}{2}\right\}\right)=1.
	$$
\end{lemma}
\begin{proof}
We define the sequence $(n_k)$ via the equation $n_k=\max\{n\geq1:g(n)^2< k\}+1$. Let the sequence $(\zeta_n)_{n=1}^{\infty}$ be as in Theorem~\ref{thm:CLT2} with $(a_n)_{n=1}^{\infty}=(\frac{\sqrt{2}g(n)}{2})_{n=1}^{\infty}$. Now applying Theorem~\ref{thm:CLT2} and using the fact that $\lim_{n\to\infty}\zeta_n=0$ we have
	\[
	\begin{split}
	\sum_{k=1}^\infty\mu\left(\left\{\iii\in\Sigma:\frac{\log\mu([\iii|_1^{n_k}])+h_\mu n_k}{\sqrt{2 \rho_\mu n_k}g(n_k)}\geq \frac{1}{2}\right\}\right)&= \sum_{k=1}^\infty e^{\frac{-g(n_k)^2(1+\zeta_{n_k})}{4}}\\
	&\ll\sum_{k=1}^\infty e^{-k/5}\\
	&<\infty.\\
	\end{split}\]
	Our claim now follows by the Borel--Cantelli Lemma (Lemma \ref{lem:BC}).
\end{proof}
Equipped with Lemma \ref{lem:sequence} we are now in a position to prove Theorem \ref{thm:nocharac}.
\begin{proof}[Proof of Theorem~\ref{thm:nocharac}]
	Let $\mathcal{G}=\{n_k:k \geq 1\}\subseteq\N$, where $(n_k)_{k=1}^{\infty}$ is the sequence defined in Lemma~\ref{lem:sequence}. Let $(a_n)_{n=1}^{\infty}$ be the sequence of natural numbers defined as follows
		$$
		a_n:=\begin{cases}
		\left\lceil e^{h_\mu n-\sqrt{2\rho_\mu n} g(n)}\right\rceil & \text{if }n\in\mathcal{G},\\[2ex]
		\left\lfloor n^{-2}e^{h_\mu n-\frac{5}{4}\sqrt{2\rho_\mu n\log\log{n}}}\right\rfloor & \text{if }n\notin\mathcal{G}.
		\end{cases}
		$$ Then let $\psi(n):=\max\{k\geq1:\sum_{m=1}^{k-1}a_m\leq n\}.$ Notice that $\#\psi^{-1}(n)=a_n$ for all $n\in\mathbb{N}$.
	
	 Thus,
\begin{align*}
\sum_{n=1}^\infty e^{-h_\mu\psi(n)+\sqrt{2\rho_\mu \psi(n)} g(\psi(n))}&=\sum_{n=1}^\infty \#\psi^{-1}(n)e^{-h_\mu n+\sqrt{2\rho_\mu n} g(n)} \\
&\geq\sum_{k=1}^\infty \#\psi^{-1}(n_k)e^{-h_\mu n_k+\sqrt{2\rho_\mu n_k} g(n_k)} \\
&\geq \sum_{k=1}^\infty1=\infty.
\end{align*}
For each $N\in\N$ we let
$$E_N^{(1)}=\left\{\iii\in\Sigma: \frac{\log\mu([\iii|_1^{n_k}])+h_\mu n_k}{\sqrt{2\rho_\mu n_k}g(n_k)}<\frac{3}{4}\text{ for every }n_k\geq N\right\}$$
and
$$E_N^{(2)}=\left\{\iii\in\Sigma: \frac{\log\mu([\iii|_1^{n}])+h_\mu n}{\sqrt{2\rho_\mu n\log\log{n}}}<\frac{5}{4}\text{ for every }n\geq N\right\}.$$
	By Theorem~\ref{thm:LIL} and Lemma~\ref{lem:sequence} we have $\mu(\bigcup_{N=1}^\infty E_N^{(1)}\cap E_N^{(2)})=1$.
	
	Let us also define
	$$
	Q(n)=\begin{cases}
		e^{-h_\mu n+\frac{3}{4}\sqrt{2\rho_\mu n} g(n)} & \text{if }n\in\mathcal{G},\\[2ex]
		e^{-h_\mu n+\frac{5}{4}\sqrt{2\rho_\mu n\log\log{n}}} & \text{if }n\notin\mathcal{G}.
	\end{cases}
	$$ It follows from the above that if $\jjj\in \cup_{n=N}^{\infty}\Sigma_n$ satisfies $[\jjj]\cap E_{N}^{(1)}\cap E_{N}^{(2)}\neq \emptyset$ then we have
	\begin{equation}
	\label{eq:Pbound}
	\mu([\jjj])\leq Q(|\jjj|).
	\end{equation}
	We will now show that $	\mu\left(E_N^{(1)}\cap E_N^{(2)}\cap R_\psi\right)=0$ for any $N\geq 1$. Since $\mu(\bigcup_{N=1}^\infty E_N^{(1)}\cap E_N^{(2)})=1$, this will complete our proof.
	
	Let us now fix $N\geq 1$. By the first Borel--Cantelli Lemma (Lemma \ref{lem:BC}), to show that $\mu\left(E_N^{(1)}\cap E_N^{(2)}\cap R_\psi\right)=0,$ it suffices to show that
	\begin{equation}
	\label{eq:WTS2}\sum_{n=1}^\infty\mu\left(\left\{\iii\in E_N^{(1)}\cap E_N^{(2)}: \sigma^n\iii\in[\iii|_1^{\psi(n)}]\right\}\right)<\infty.
	\end{equation}
	Let us start by defining $M_{N}\in\mathbb{N}$ to be sufficiently such that for all $n\geq M_{N}$ we have $\psi(n)\geq N$. We now show that \eqref{eq:WTS2} holds. We first note that it follows from \eqref{eq:simpleform} and the fact that $\mu$ is a Bernoulli measure, so \eqref{Bernoulli} holds, that
\begin{align*}
\sum_{n=1}^\infty\mu\left(\left\{\iii\in E_N^{(1)}\cap E_N^{(2)}: \sigma^n\iii\in[\iii|_1^{\psi(n)}]\right\}\right) 					
                &\leq M_{N}+\sum_{n=M_{N}}^\infty\sum_{\substack{\jjj \in \Sigma_{\psi(n)} \\ \iii \in \Sigma_{n-\psi(n)} \\ [\jjj]\cap E_N^{(1)} \cap E_N^{(2)} \neq \emptyset}}{\mu([\jjj\iii\jjj])}\\
                &= M_{N}+\sum_{n=M_{N}}^\infty\sum_{\substack{\jjj \in \Sigma_{\psi(n)} \\ \iii \in \Sigma_{n-\psi(n)} \\ [\jjj]\cap E_N^{(1)} \cap E_N^{(2)} \neq \emptyset}}{\mu([\jjj])\mu([\iii])\mu([\jjj])}\\
				&\leq M_{N}+\sum_{n=M_{N}}^\infty\sum_{\substack{\jjj\in\Sigma_{\psi(n)} \\ [\jjj]\cap E_N^{(1)}\cap E_N^{(2)}\neq\emptyset}}\mu([\jjj])^2.
\end{align*}
Next, from \eqref{eq:Pbound} and the fact that $\psi(n)\geq N$ for all $n \geq M_N$, we see that
\begin{align*}
\sum_{n=1}^\infty\mu\left(\left\{\iii\in E_N^{(1)}\cap E_N^{(2)}: \sigma^n\iii\in[\iii|_1^{\psi(n)}]\right\}\right)
		&\leq M_{N}+C\sum_{n=M_{N}}^\infty\sum_{\substack{\jjj\in\Sigma_{\psi(n)} \\ [\jjj]\cap E_N^{(1)}\cap E_N^{(2)}\neq\emptyset}}Q(\psi(n))\mu([\jjj]).
\end{align*}
Finally, recalling the definitions of $Q(n)$, $a_n$, and the fact that $\#\psi^{-1}(n)=a_n$ for all $n \in \N$, we observe that
\begin{align*}
&\sum_{n=M_{N}}^\infty\sum_{\substack{\jjj\in\Sigma_{\psi(n)} \\ [\jjj]\cap E_N^{(1)}\cap E_N^{(2)}\neq\emptyset}}Q(\psi(n))\mu([\jjj])	\\	
		&\leq \sum_{n=1}^\infty \#\psi^{-1}(n)Q(n)\\
		&\phantom{}=\sum_{n\in\mathcal{G}}\left\lceil e^{h_\mu n-\sqrt{2\rho_\mu n} g(n)}\right\rceil e^{-h_\mu n+\frac{3}{4}\sqrt{2\rho_\mu n} g(n)}+\sum_{n\notin\mathcal{G}}\left\lfloor n^{-2}e^{h_\mu n-\frac{5}{4}\sqrt{2\rho_\mu n\log\log(n)}}\right\rfloor e^{-h_\mu n+\frac{5}{4}\sqrt{2\rho_\mu n\log\log{n}}}\\
		&\ll \sum_{n=1}^\infty e^{-\frac{1}{4} \sqrt{2\rho_{\mu}n}g(n)}+\sum_{n=1}^\infty n^{-2}<\infty.
\end{align*}
Thus it follows that \eqref{eq:WTS2} holds and our proof is complete.
\end{proof}

\section{Applications to dynamics on self-similar sets}
\label{Sec:self-similar}
In this section we apply our results to the study of dynamics on self-similar sets. Before that it is necessary to define some preliminary notions. We call a map $\varphi:\mathbb{R}^d\to\mathbb{R}^d$ a contracting similarity if there exists $r\in (0,1)$ such that 
\[\|\varphi(x)-\varphi(y)\|=r\|x-y\| \quad \text{for all } x,y\in\mathbb{R}^{d}.\] 
In this section, we define an {\it iterated function system, or IFS} for short, to be a finite set of contracting similarities $\Phi := \{\phi_i\}_{i=1}^{K}$. When each similarity in an IFS has the same contraction ratio the IFS is said to be \emph{homogeneous}. A well known result due to Hutchinson \cite{Hut} states that for any IFS there exists a unique, non-empty, compact set $X\subset\mathbb{R}^d$ satisfying $$\bigcup_{i=1}^{K}\varphi_{i}(X)=X.$$ The set $X$ is called the \emph{self-similar set} or \emph{attractor} of $\Phi$. Self-similar sets are important and well studied objects in the field of Fractal Geometry. For more on these sets we refer the reader to Falconer's book \cite{Fal}. Self-similar sets can be viewed as the image of $\Sigma=\{1,\ldots,K\}^{\N}$ under an appropriate map. Let $\pi:\Sigma\to X$ be given by 
\[\pi(\iii)=\lim_{n\to\infty}\left(\varphi_{i_1}\circ \cdots \circ \varphi_{i_n}\right)(0).\]
The map $\pi$ is continuous and surjective. Importantly $\pi$ allows us to take measures defined on $\Sigma$, most notably Gibbs measures, and to project them forward on to $X$. The measures defined on $X$ as pushforwards under $\pi$ are an important and well studied class. The most well studied measures amongst this class are the self-similar measures which are the pushforwards of Bernoulli measures.

If $\Phi$ satisfies the additional assumption that 
\[\varphi_i(X)\cap \varphi_{j}(X)=\emptyset \quad \text{for all } i\neq j\] 
then $\Phi$ is said to satisfy the \emph{strong separation condition}. This condition is equivalent to the map $\pi$ being a bijection. If an IFS satisfies the strong separation condition we can define a map $T:X\to X$ according to the rule 
\[T(x)=\varphi_{i}^{-1}(x) \quad \text{if} \quad x\in \varphi_i(X).\] 
Because $X=\bigcup_{i=1}^{K}\varphi_{i}(X)$ and this union is disjoint, the map $T$ is well defined.

Given a homogeneous IFS with common contraction ratio $r$ which satisfies the strong separation condition and a function $\psi:\N\to [0,\infty),$ we can define a recurrence set as follows:
$$\tilde{R}_{\psi}:=\left\{x\in X:\|T^{n}(x)-x\|\leq r^{\psi(n)}\textrm{ for infinitely many }n\in\mathbb{N}\right\}.$$ Here $\|\cdot\|$ denotes the Euclidean norm. This family of recurrence sets was studied previously in \cite{BakFar} and \cite{ChangWuWu}. Combining these papers with the mass transference principle of Beresnevich and Velani \cite{BerVel}, it is possible to obtain a detailed description of the metric properties of $\tilde{R}_{\psi}$ in terms of Hausdorff measure. As we will see, the results of the current paper allow us to prove new statements on the $\pi^* \mu$ measure of $\tilde{R}_{\psi}$. Here $\pi^*\mu$ denotes the pushforward of some Gibbs measure $\mu$ under the map $\pi$; that is, $\pi^*\mu=\mu \circ \pi^{-1}$. The key proposition that allows us to translate our previous results for $R_{\psi}$ (recall the definition in \eqref{eq:recset}) into statements for $\tilde{R}_{\psi}$ is the following:
\begin{proposition}
	\label{Prop:inclusions}
Let $\Phi$ be a homogeneous IFS satisfying the strong separation condition with associated self-similar set $X$ and let $\psi:\N\to [0,\infty)$. There exists $N\in\mathbb{N}$ depending only upon $\Phi$ such that we have the following inclusions:
$$\pi(R_{\lfloor \psi \rfloor +N})\subseteq \tilde{R}_{\psi}\subseteq \pi(R_{\lfloor \psi \rfloor -N}).$$
\end{proposition}

\begin{proof}
This follows from the observation that $\pi(\sigma \iii)=T(\pi \iii)$ together with the following fact from Fractal Geometry.\\

\noindent \textbf{Fact:} Let $x,y\in X$ and $\iii,\jjj\in \Sigma$ be such that $\pi(\iii)=x$ and $\pi(\jjj)=y$. Then there exists $N\in\mathbb{N}$ such that:
\begin{equation}
\label{Fact1}
\text{if }\|x-y\|\leq r^{\psi(n)}\text{ then }|\iii\wedge \jjj|\geq \lfloor \psi(n)\rfloor-N
\end{equation} and
\begin{equation}
\label{Fact2}
\text{if }|\iii\wedge \jjj|\geq \lfloor \psi(n)\rfloor+N \text{ then }\|x-y\|\leq r^{\psi(n)}.
\end{equation}\\
We include a proof of this fact for completion.\\

Let $N\in\mathbb{N}$ be sufficiently large that $$\inf_{i,j: i\neq j}d(\varphi_{i}(X),\varphi_{j}(X))\geq r^{N}\text{ and }r^{N-1}\cdot Diam(X)<1.$$ Here $d$ is the Euclidean metric. Note that $\inf_{i,j: i\neq j}d(\varphi_{i}(X),\varphi_{j}(X))>0$ because of the strong separation condition. Therefore $N$ is well defined. We now prove that this $N$ satisfies the desired properties.\\

Let $x,y\in X$ be such that $\|x-y\|\leq r^{\psi(n)}$. By considering inverses and the maximal common prefix $\iii\wedge \jjj,$ we see that 
$$r^{-|\iii\wedge \jjj|}\|x-y\|=\|(\varphi_{i_1}\circ \cdots \circ \varphi_{i_{|\iii\wedge \jjj|}})^{-1}(x)-(\varphi_{j_1}\circ \cdots \circ \varphi_{j_{|\iii\wedge \jjj|}})^{-1}(y)\|\geq r^{N}.$$ 
Therefore $\|x-y\|\geq r^{N+|\iii\wedge \jjj|}$ and we have $\psi(n)\leq N+|\iii\wedge \jjj|.$ Taking integer parts we see that $|\iii\wedge \jjj|\geq \lfloor\psi(n)\rfloor-N$ and \eqref{Fact1} holds.

Now suppose $|\iii\wedge \jjj|\geq \lfloor \psi(n)\rfloor+N.$ Then \begin{align*}
\|x-y\|&\leq Diam((\varphi_{i_1}\circ \cdots \circ \varphi_{i_{|\iii\wedge \jjj|}})(X)) \\
       &\leq r^{|\iii\wedge \jjj|}\cdot Diam(X) \\
       &\leq r^{\lfloor \psi(n)\rfloor+N}\cdot Diam(X) \\
       &\leq r^{\psi(n)-1+N}\cdot Diam(X) \\
       &\leq r^{\psi(n)}.
\end{align*} Therefore \eqref{Fact2} holds and our proof is complete.
\end{proof}
Equipped with Proposition \ref{Prop:inclusions}, it is possible to translate Theorems \ref{thm:convergence}, \ref{thm:meas1}, \ref{loglogtheorem}, and \ref{thm:nocharac} into the setting of dynamics on self-similar sets and pushforwards of Gibbs measures. The key point that allows us to establish these analogues is that the parameter $N,$ whose existence is asserted by Proposition \ref{Prop:inclusions}, only depends upon the underlying IFS. As such, if $\psi:\N\to [0,\infty)$ satisfies some appropriate hypothesis analogous to that given in one of the theorems listed above, then the functions $\lfloor \psi(n)\rfloor-N$ and $\lfloor \psi(n)\rfloor+N$ will also satisfy the hypothesis formulated in the original theorem. Using this observation together with the inclusions given in Proposition \ref{Prop:inclusions}, we may prove appropriate analogues of Theorems \ref{thm:convergence}, \ref{thm:meas1}, \ref{loglogtheorem}, and \ref{thm:nocharac} in this setting. For the sake of brevity we do not give the statement of each of these analogues here. We instead content ourselves with the following analogue of Theorem \ref{loglogtheorem}.

\begin{theorem}
\label{thm:ssloglogtheorem}
Let $\Phi$ be a homogeneous IFS satisfying the strong separation condition and let $\mu$ be the Gibbs measure of a H\"{o}lder continuous potential that is not cohomologous to a constant. For $\varepsilon>0$ let $\psi_{\varepsilon}^{+}:\mathbb{N}\to [0,\infty)$ and $\psi_{\varepsilon}^{-}:\mathbb{N}\to [0,\infty)$ be given by $$\psi_{\varepsilon}^{+}(n)= \frac{\log n}{h_{\mu}}+\frac{(1+\varepsilon)}{h_{\mu}^{3/2}}\sqrt{2\rho_{\mu}\log n\log \log \log n}$$ and $$\psi_{\varepsilon}^{-}(n)= \frac{\log n}{h_{\mu}}+\frac{(1-\varepsilon)}{h_{\mu}^{3/2}}\sqrt{2\rho_{\mu}\log n\log \log \log n}.$$ Then for any $\varepsilon>0$ we have $\pi^*\mu(\tilde{R}_{\psi_{\varepsilon}^{+}})=0$ and $\pi^*\mu(\tilde{R}_{\psi_{\varepsilon}^{-}})=1.$
\end{theorem}
Theorem \ref{thm:ssloglogtheorem} demonstrates that the critical threshold observed for Gibbs measures and shifts of finite type persists in the setting of dynamics on self-similar sets. Theorem \ref{thm:ssloglogtheorem} also yields some interesting metric properties for the sets $\tilde{R}_{\psi_{\varepsilon}^{-}}$ that do not follow from \cite{BakFar} or \cite{ChangWuWu}.It is well known that for any Gibbs measure $\mu$ for which the defining potential is not cohomologous to a constant we have $h_\mu<\log K$, and so by a covering argument, $\mathcal{H}^{\dimh{X}}(\tilde{R}_{\psi_{\varepsilon}^{-}})=0$ for any $\varepsilon\in(0,1)$. Here $\dimh{X}$ is the Hausdorff dimension of the self-similar set $X$. Moreover, using the mass transference principle of Beresnevich and Velani \cite{BerVel}, it is possible to show that $\dimh{(\tilde{R}_{\psi_{\varepsilon}^{-}})}=\dimh{X}$ for any $\varepsilon>0$. Therefore, despite being null in terms of the $\dimh{(\tilde{R}_{\psi_{\varepsilon}^{-}})}$-dimensional Hausdorff measure, Theorem \ref{thm:ssloglogtheorem} tells us that $\tilde{R}_{\psi_{\varepsilon}^{-}}$ is large in terms of $\mu$ for any $\varepsilon>0$.

\end{document}